\documentclass[12pt]{article}
\usepackage{amsthm}
\usepackage{amsmath,amsfonts,amssymb,rotating,colordvi}
\usepackage{color}
\usepackage[unicode=true,
 bookmarks=true,bookmarksnumbered=true,bookmarksopen=true,bookmarksopenlevel=2,
 breaklinks=false,pdfborder={0 0 1},backref=false,colorlinks=true]
 {hyperref}

\usepackage{tikz}

\usetikzlibrary{calc}
\usetikzlibrary{shapes.geometric}

\tikzstyle dot=[style={circle,inner sep=1pt,fill}]

\tikzset{
  c/.style={every coordinate/.try}
}

\usetikzlibrary{arrows,shapes,positioning}
\usetikzlibrary{decorations.markings}
\tikzstyle arrowstyle=[scale=1]
\tikzstyle directed=[postaction={decorate,decoration={markings,mark=at position 0.6 with {\arrow[arrowstyle]{stealth};}}}]
\tikzstyle reverse directed=[postaction={decorate,decoration={markings,mark=at position 0.4 with {\arrowreversed[arrowstyle]{stealth};}}}]
\tikzstyle dot=[style={circle,inner sep=1pt,fill}]
\hypersetup{pdftitle={Word-representability of subdivisions of triangular grid graphs},
 pdfauthor={Zongqing Chen, Sergey Kitaev and Brian Y. Sun,...},
 pdfsubject={ },
 pdfkeywords={ },
 linkcolor=black,citecolor=black,urlcolor=blue,filecolor=blue,pdfpagelayout=OneColumn,pdfnewwindow=true,pdfstartview=XYZ,plainpages=false}

\def\qed{\nopagebreak\hfill{\rule{4pt}{7pt}}}
\def\ps{\mathrm{ps}}
 \newtheorem{thm}{Theorem}[section]

\newtheorem{lem}[thm]{Lemma}

\numberwithin{equation}{section}

\renewcommand{\theequation}{\arabic{section}.\arabic{equation}}
\renewcommand{\thetable}{\arabic{section}.\arabic{table}}
\renewcommand{\thefigure}{\arabic{section}.\arabic{figure}}
\newdimen\Squaresize \Squaresize=11pt
\newdimen\Thickness \Thickness=0.7pt
\def\Square#1{\hbox{\vrule width \Thickness
   \vbox to \Squaresize{\hrule height \Thickness\vss
    \hbox to \Squaresize{\hss#1\hss}
   \vss\hrule height\Thickness}
\unskip\vrule width \Thickness} \kern-\Thickness}

\def\Vsquare#1{\vbox{\Square{$#1$}}\kern-\Thickness}
\def\blk{\omit\hskip\Squaresize}

\def\young#1{
\vbox{\smallskip\offinterlineskip \halign{&\Vsquare{##}\cr #1}}}
\def\putd#1{\scriptsize{-#1-}}
\def\moins{\raise 1pt\hbox{{$\scriptstyle -$}}}

\newcommand\p{\circle*{0.3}}

\begin{document}

\begin{center}
\textbf{\large{Word-representability of triangulations of grid-covered cylinder graphs}}\textbf{ }
\par\end{center}

\begin{center}
Thomas Z.Q. Chen$^{a}$, Sergey Kitaev$^{b}$, Brian Y. Sun$^{c}$\\[6pt]
\par\end{center}

\begin{center}
$^{a,c}$Center for Combinatorics, LPMC-TJKLC\\
 Nankai University, Tianjin 300071, P. R. China\\
$^{b}$Department of Computer and Information Sciences,\\
 University of Strathclyde, Glasgow, G1 1XH, UK
\par\end{center}

\begin{center}
Email: $^{a}$\texttt{zqchern@163.com}, $^{b}$\texttt{sergey.kitaev@cis.strath.ac.uk},
$^{c}$\texttt{brian@mail.nankai.edu.cn}
\par
\end{center}

\par
\noindent \textbf{Abstract.}
A graph $G=(V,E)$ is word-representable if there exists a word $w$
over the alphabet $V$ such that letters $x$ and $y$, $x\neq y$, alternate in $w$ if and only if $(x,y)\in E$. Halld\'{o}rsson et al.\ have shown that a graph is word-representable if and only if it admits a so-called semi-transitive orientation. A corollary to this result is that any 3-colorable graph is word-representable.

Akrobotu et al.\ have shown that a triangulation of a grid graph is word-representable if and only if it is 3-colorable. This result does not hold for triangulations of grid-covered cylinder graphs, namely, there are such word-representable graphs with chromatic number 4. In this paper we show that word-representability of triangulations of grid-covered cylinder graphs with three sectors (resp., more than three sectors) is characterized by avoiding a certain set of six minimal induced subgraphs (resp., wheel graphs $W_5$ and $W_7$).\\


\section{Introduction}

Let $G=(V,E)$ be a simple (i.e. without loops and multiple edges) undirected graph with
the vertex set $V$ and the edge set $E$. We say that $G$ is {\em word-representable} if there exists a word $w$
over the alphabet $V$ such that letters $x$ and $y$ alternate in $w$ if and only if $(x,y)\in E$ for
any $x\neq y$. By definition, each letter in $V$ must appear in $w$.

The notion of word-representable graphs has its roots in algebra,
where a prototype of these graphs was used by Kitaev and Seif to study
the growth of the free spectrum of the well known
{\em Perkins semigroup} \cite{Kitaev08b}.

A number of results on word-representable graphs
were obtained in the literature \cite{Akrobotu, CKS, Collins14, GK, Halldorsson11, Halldorsson10,Halldorsson15, Kitaev08a, Kitaev11c, Kitaev11}. In
particular, Halld\'{o}rsson et al.~\cite{Halldorsson15} have shown that a graph is word-representable
if and only if it admits a {\em semi-transitive orientation} (to be defined in Section~\ref{sec2}), which, among other important corollaries, implies that all 3-colorable graphs are word-representable. We refer to the upcoming book~\cite{KitLoz} for the state of the art in the theory of word-representable graphs.

Most relevant to our paper are \cite{Akrobotu,CKS,GK}, where {\em triangulations} and {\em subdivisions} of certain graphs are studied with respect to word-representability. In particular, Akrobotu et al.~\cite{Akrobotu} proved that any triangulation of the graph $G$ associated with a convex polyomino is word-representable if and only if $G$ is 3-colorable. The method to prove this characterization theorem was essentially in showing that such a triangulation is 3-colorable if and only if it contains no wheel graph $W_5$ or $W_7$ as an induced subgraph (neither $W_5$ no $W_7$ are word-representable).

In this paper we extend the results of Akrobotu et al.~\cite{Akrobotu} to the case of grid-covered cylinder graphs, which is a cyclic version of rectangular grid graphs; see Subsection~\ref{GCC-sub} for definitions. It turns out that in this case, some of the graphs in question with chromatic number 4 are actually word-representable; for example, see the underlying graph in Figure~\ref{ex-semi-tran-O}. Still, assuming that there are at least four sectors in a grid-covered cylinder graph, word-representable triangulations of such graphs are characterized by avoiding $W_5$ and $W_7$ as induced subgraphs. On the other hand, we can also characterized word-representability of triangulations of grid-covered cylinder with three sectors as those avoiding the six graphs in Figure~\ref{non-repr-induced-subgraphs} as induced subgraphs. Moreover, we show that our characterization results in the case of more than three sectors hold even when some of cells (faces) of grid-covered cylinder graphs are not triangulated.

The paper is organized as follows. In Section~\ref{sec2} we will provide all necessary definitions and known results to be used. In particular, we will introduce the notion of a triangulation of a grid-covered cylinder graph, the main concern of this paper. Also, we will introduce the notion of a semi-transitive orientation, the main tool to prove our results. Further, we classify word-representable triangulations of graphs in question depending on the number of sectors they have. Namely, in Sections~\ref{sec3} we will consider the case of grid-covered cylinder graphs with more than three sections, and in Section~\ref{sec4} we will consider the case of grid-covered cylinder graphs with three sections. Finally, in Section~\ref{sec5} we discuss a generalization of our results and state an open problem.

\section{Definitions, notation, and known results}\label{sec2}

Suppose that $w$ is a word and $x$ and $y$ are two distinct letters in $w$.
We say that $x$ and $y$ {\it alternate} in $w$ if the deletion of all other
letters from the word $w$ results in either $xyxy\cdots$ or $yxyx\cdots$.

A graph $G=(V,E)$ is {\it word-representable} if there exists a word $w$ over
the alphabet $V$ such that letters $x$ and $y$ alternate in $w$ if and only if
$(x,y)\in E$ for each $x\neq y$. We say that {\it w represents G}, and such a word $w$
is called a {\it word-representant} for $G$.
For example, if the word $w=134231241$ then the subword induced with letters $1$ and $2$
is $12121$, hence the letters $1$ and $2$ alternate in $w$, and thus the respective vertices are connected in $G$. On the other hand, the letters $1$ and $3$
do not alternate in $w$, because removing all other letters we obtain $1331$; thus, $1$ and $3$ are not connected in $G$. Figure~\ref{wordandgraph} shows the graph represented by $w$.

\subsection{Semi-transitive orientations}

A directed graph (digraph) is {\it semi-transitive} if it is acyclic (that is, it contains no directed cycles), and
for any directed path $v_1\rightarrow v_2\rightarrow\cdots\rightarrow v_k$ with $v_i\in V$
for all $i, 1\leq i \leq k$, either
\begin{enumerate}
 \item[$\bullet$] there is no edge $v_1\rightarrow v_k$, or
 \item[$\bullet$] the edge $v_1\rightarrow v_k$ is present and there are edges $v_i\rightarrow v_j$ for all
     $1\leq i< j \leq k$. That is, in this case, the (acyclic) subgraph induced by the vertices
     $v_1,\ldots,v_k$ is transitive.
\end{enumerate}
We call such an orientation a {\em semi-transitive orientation}.

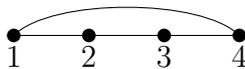
\begin{figure}[!htbp]
 \begin{center}
\begin{tikzpicture}
 \draw (0,0) node[below]{1} -- ++(1,0)node[below]{2} -- ++(1,0)node[below]{3} -- ++(1,0)node[below]{4};
 \draw (0,0) .. controls (0.5,.5)and (2.5,.5).. (3,0);
 \fill[black!100] (0,0) circle(0.5ex)
                 ++(1,0) circle(0.5ex)
                 ++(1,0) circle(0.5ex)
                 ++(1,0) circle(0.5ex);
\end{tikzpicture}
\caption{\label{wordandgraph} The graph represented by the word $w=134231241$}
\end{center}
\end{figure}

We can alternatively define semi-transitive orientations in terms of induced subgraphs. A \emph{semi-cycle} is the directed
acyclic graph obtained by reversing the direction of one arc of a directed cycle. An acyclic digraph is a
\emph{shortcut} if it is induced by the vertices of a semi-cycle and contains a pair of non-adjacent vertices. Thus, a
digraph on the vertex set $\{ v_1, \ldots, v_k\}$ is a shortcut if it contains a directed path $v_1\rightarrow
v_2\rightarrow \cdots \rightarrow v_k$, the arc $v_1\rightarrow v_k$, and it is missing an arc $v_i\rightarrow v_j$
for some $1 \le i < j \le k$; in particular, we must have $k\geq 4$, so that any shortcut is on at least four
vertices. Slightly abusing the terminology, in this paper we refer to the arc  $v_1\rightarrow v_k$ in the last definition as a shortcut (a more appropriate name for this would be a {\em shortcut arc}). Figure~\ref{shortcut} gives examples of shortcuts, where the edges $1\rightarrow 4$, $2\rightarrow 5$ and $3\rightarrow 6$ are missing, and hence $1\rightarrow 5$, $1\rightarrow 6$ and $2\rightarrow 6$ are shortcuts.

Thus, an orientation of a graph is semi-transitive if it is acyclic and contains no shortcuts. Halld\'{o}rsson et al.\ \cite{Halldorsson15} proved the following theorem that characterizes  word-representable graphs in terms of graph orientations.

\begin{figure}[!htbp]
 \begin{center}
\begin{tikzpicture}
 \draw (0,0) node[below]{1} -- ++(1,0)node[below]{2} -- ++(1,0)node[below]{3} -- ++(1,0)node[below]{4}
             -- ++(1,0)node[below]{5} -- ++(1,0)node[below]{6};
 \foreach \i in {0,1,...,4}
 { \draw[,directed] (\i,0)--+(1,0);
 }
 \draw[,directed] (0,0) .. controls (0.5,.5)and (1.5,.5).. (2,0);
  \draw[,directed] (0,0) .. controls (0.5,-.5)and (3.5,-.5).. (4,0);
   \draw[,directed] (0,0) .. controls (0.5,.85)and (4.5,.85).. (5,0);
    \draw[,directed] (1,0) .. controls (1.5,.5)and (2.5,.5).. (3,0);
     \draw[,directed] (1,0) .. controls (1.5,-.75)and (4.5,-.75).. (5,0);
      \draw[,directed] (2,0) .. controls (2.5,.5)and (3.5,.5).. (4,0);
       \draw[,directed] (3,0) .. controls (3.5,.5)and (4.5,.5).. (5,0);
 \fill[black!100] (0,0) circle(0.5ex)
                 ++(1,0) circle(0.5ex)
                 ++(1,0) circle(0.5ex)
                 ++(1,0) circle(0.5ex)
                 ++(1,0) circle(0.5ex)
                 ++(1,0) circle(0.5ex);
\end{tikzpicture}

\caption{ An example of a shortcut}\label{shortcut}
\end{center}
\end{figure}
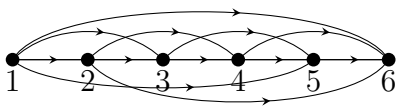

\begin{thm}\label{semitra}
 {\em \cite{Halldorsson11}} A graph is word-representable if and only if it admits a semi-transitive orientation.
\end{thm}

Thus, in this paper, to find out if a graph $G$ in question is word-representable, we will be studying existence of a semi-transitive orientation on $G$.

An immediate corollary to Theorem~\ref{semitra} is the following result.

\begin{thm}{\label{color}} {\em \cite{Halldorsson11}} Three-colorable graphs are word-representable.
\end{thm}

\subsection{Grid-covered cylinder graphs}\label{GCC-sub}

A {\em grid-covered cylinder}, {\em GCC} for brevity, is a 3-dimensional figure formed by drawing vertical lines and horizontal circles on the surface of a cylinder, each of which are parallel to the generating line and the upper face of the cylinder, respectively. A GCC can be thought of as the object obtained by gluing the left and right sides of a rectangular grid. See the left picture in Figure~\ref{grid-covered-cylind-pic} for a schematic way to draw a GCC.  The vertical lines and horizontal circles are called the {\em grid lines} by us.

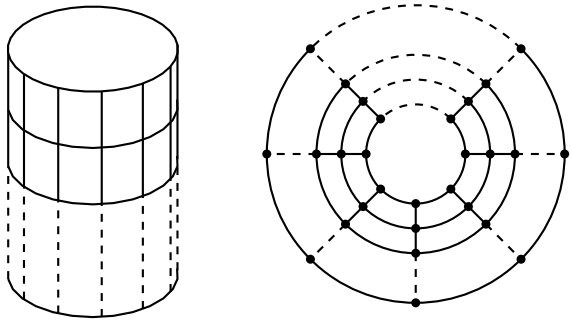
\begin{figure}[!htbp]
 \begin{center}
\begin{tikzpicture}[scale=1,x={(-0.8cm,-0.4cm)},y={(0.8cm,-0.4cm)},
    z={(0cm,1cm)},font=\large]
	\def\h{1.5}

	\foreach \t in {-39,-9,...,154}
		\draw[thick] ({cos(\t)},{sin(\t)},{1*\h}) --({cos(\t)},{sin(\t)},{2.0*\h});
	\foreach \t in {-39,-9,...,154}
		\draw[thick,dashed] ({cos(\t)},{sin(\t)},0) --({cos(\t)},{sin(\t)},{1*\h});

	\draw[,thick] ({cos(-39)},{sin(-39)},0) 
		\foreach \t in {-39,-25,...,154}
			{--({cos(\t)},{sin(\t)},0)};
			
	\draw[,thick] (1,0,{2*\h}) 
		\foreach \t in {10,20,...,360}
			{--({cos(\t)},{sin(\t)},{2*\h})}--cycle;
			
	\draw[,thick] ({cos(-39)},{sin(-39)},{1.5*\h}) 
		\foreach \t in {-39,-25,...,154}
			{--({cos(\t)},{sin(\t)},{1.5*\h})};
				
	\draw[,thick] ({cos(-39)},{sin(-39)},{1*\h}) 
		\foreach \t in {-39,-25,...,154}
			{--({cos(\t)},{sin(\t)},{1*\h})};
		
\end{tikzpicture}
\qquad
\begin{tikzpicture} [scale=.66, line join=round, >=latex,thick]
\foreach \r in {1,1.5,2,3}
{
\draw[dashed] ([shift=(45:\r cm)]0,0) arc (45:135:\r cm);
\draw[] ([shift=(45:\r cm)]0,0) arc (45:-225:\r cm);
}
\foreach \a in {-2,-1,0,1,-3,-4,-5}
{
\draw (45*\a:1 cm)-- (45*\a:2 cm);
\draw[dashed] (45*\a:2 cm)-- (45*\a:3 cm);
 \fill[black!100] (45*\a:1 cm) circle(0.5ex)
                 (45*\a:1.5 cm) circle(0.5ex)
                 (45*\a:2 cm) circle(0.5ex)
                 (45*\a:3 cm) circle(0.5ex)
                ;
}
\end{tikzpicture}
\caption{Grid-covered cylinder}\label{grid-covered-cylind-pic}
\end{center}
\end{figure}

Any GCC defines a graph whose set of vertices is given by intersection of the grid lines, and whose edges are parts of grid lines between the respective vertices. {\em Vertical edges} and {\em horizontal edges} are defined by vertical and horizontal grid lines, respectively. Such a graph is necessarily planar, and it is convenient to consider its edge-crossing-free embedding in the plane as shown schematically in the right picture in Figure~\ref{grid-covered-cylind-pic}, where by convention, the internal circle $C_0$ corresponds to the top face of the respective GCC.

We next introduce some notions/notation related to a GCC graph (abbreviated {\em GCCG}) $G_{m,n}$ defined by intersection of $m$ vertical lines and $n+1$ horizontal circles.  Let $C_0, C_1, \ldots, C_n$ denote the circles of $G_{m,n}$ in order from inside to outside. Further, let $C_i$, for $0\leq i\leq n$, have $m$ equally spaced vertices denoted by $v_{i0},v_{i1},\ldots,v_{i(m-1)}$ in the counter-clock-wise direction, so that for a fixed $y$ and any $x$, vertices $v_{xy}$ lie on the same vertical grid line labelled by $V_y$; see Figure~\ref{GCC-4-labelled} for an example of a proper labelling of a GCCG with four sectors.

 \begin{figure}[!htbp]
 \begin{center}
\begin{tikzpicture} [scale=0.8, line join=round, >=latex,thick]
\def\ra{1} 
\def\rb{1.93} 
\def\rc{2.85} 
\foreach \r in {0.7,1.6,2.5}
  \draw (0,0) circle (\r) ;
\foreach \a in {0,1,2,3}
{
\draw (90*\a:0.7 cm)-- (90*\a:2.5 cm);
 \fill[black!100] (90*\a:.7 cm) circle(0.5ex)
                 (90*\a:1.6 cm) circle(0.5ex)
                 (90*\a:2.5 cm) circle(0.5ex)
                ;
}
\path ( 0:\ra cm) node[above] {$v_{01}$}
      (90:\ra cm) node[left] {$v_{02}$}
      (180:\ra cm) node[above] {$v_{03}$}
      (-90:\ra cm) node[left] {$v_{00}$}
      ;
\path ( 0:\rb cm) node[above] {$v_{11}$}
      (90:\rb cm) node[left] {$v_{12}$}
      (180:\rb cm) node[above] {$v_{13}$}
      (-90:\rb cm) node[left] {$v_{10}$}
      ;
\path ( 0:\rc cm) node[above] {$v_{21}$}
      (90:\rc cm) node[left] {$v_{22}$}
      (180:\rc cm) node[above] {$v_{23}$}
      (-90:\rc cm) node[left] {$v_{20}$}
      ;
\end{tikzpicture}

\caption{Labelling of a GCCG}\label{GCC-4-labelled}
\end{center}
\end{figure}
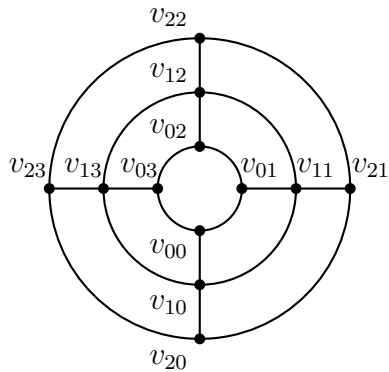

We say that the vertices on a circle $C_i$ are on the {\em $i$th level}. Also, we say that $G_{m,n}$ has $n$ {\em layers} and $m$ {\em sectors}. For $1\leq i\leq n$, the $i$th layer $L_i$ is the induced subgraph of  $G_{m,n}$ formed by the vertices on $C_{i-1}$ and $C_i$. For $1\leq j\leq m-1$, the $j$th sector $S_j$ is the induced subgraph formed by the vertices on the $(j-1)$th and $j$th vertical grid lines (i.e. $V_{j-1}$ and $V_j$); the $m$th sector $S_m$ is  the induced subgraph formed by the vertices on the $(m-1)$th and $0$th vertical grid lines (i.e. $V_{m-1}$ and $V_0$). For example, referring to Figure~\ref{GCC-4-labelled}, the layer $L_2$ is the induced subgraph formed by the vertices $\{v_{10}, v_{11},v_{12},v_{13},v_{20},v_{21},v_{22}, v_{23}\}$, while the sector $S_3$ is the induced subgraph formed by the vertices $\{v_{02}, v_{12},v_{22},v_{03},v_{13},v_{23}\}$.

 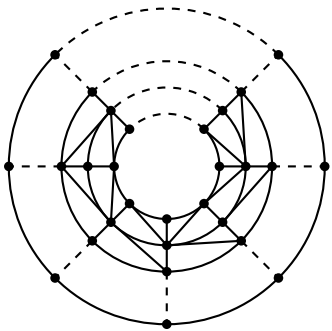
\begin{figure}[!htbp]
 \begin{center}
\begin{tikzpicture} [scale=0.7, line join=round, >=latex, thick]
\foreach \r in {1,1.5,2,3}
{
\draw[dashed] ([shift=(45:\r cm)]0,0) arc (45:135:\r cm);
\draw[] ([shift=(45:\r cm)]0,0) arc (45:-225:\r cm);
}
\foreach \a in {-2,-1,0,1,-3,-4,-5}
{
\draw (45*\a:1 cm)-- (45*\a:2 cm);
\draw[dashed] (45*\a:2 cm)-- (45*\a:3 cm);
 \fill[black!100] (45*\a:1 cm) circle(0.5ex)
                 (45*\a:1.5 cm) circle(0.5ex)
                 (45*\a:2 cm) circle(0.5ex)
                 (45*\a:3 cm) circle(0.5ex)
                ;
}

\draw (45:1 cm)--(0:1.5 cm)--(45:2 cm)
      (0:1.5 cm)--(-45:1 cm)
      (0:2 cm)--(-45:1.5 cm)
      ;
\draw (-45:1 cm)--(-90:1.5 cm)--(-45:2 cm)
      (-90:1.5 cm)--(-135:1 cm)
      (-90:2 cm)--(-135:1.5 cm)
      --(-180:1 cm)--(-225:1.5 cm)-- (-180:2 cm)--(-135:1.5 cm)
      ;
\end{tikzpicture}
\caption{A triangulation of a GCCG}\label{triang-grid-cover-cylinder-general}
\end{center}
\end{figure}

Intersections of grid lines define GCCG's {\em cells}  all of which are 4-cycles. Note that in the case of $m=4$, the vertices and edges on $C_0$ and $C_m$ form 4-cycles, but we do not call these cells. Thus, by definition, $G_{m,n}$ has $mn$ cells.  A {\em triangulation} $T$ of $G_{m,n}$ is the graph obtained from $G_{m,n}$ by triangulating {\em each} cell in it. The total number of (possibly isomorphic) triangulations of $G_{m,n}$ is $2^{mn}$. The subdivision edges in $T$ that are used to subdivide cells into triangles are called {\em diagonal edges}.

\section{Word-representable triangulations of GCCGs with more than three sectors}\label{sec3}

For $n\geq 3$, a {\em wheel graph} $W_n$ is obtained by adding to the {\em cycle graph} $C_n$ an all adjacent vertex. It is known~\cite{KitLoz,Kitaev08a} that for odd $n\geq 5$, $W_n$ is not word-representable. In particular, $W_5$ and $W_7$ shown in bold in Figure~\ref{just-4-cases} (each twice) are not word-representable.

In this section we will prove the following theorem.

\begin{thm}\label{thm-morethan-3} A triangulation of a GCCG with more than three sectors is word-representable if and only if it contains no $W_5$ or $W_7$ as an induced subgraph. \end{thm}

The proof of Theorem~\ref{thm-morethan-3} will follow from Lemma~\ref{orient-semi-trans-lem} below, whose proof is based on Lemma~\ref{structure-lem} giving the structure of GCCGs with more than three sectors that contain no $W_5$ or $W_7$ as induced subgraphs.

Next, we introduce the notion of  a {\em type} of a cell subdivision on layer $L_i$ for $i\geq 2$. We say that a cell $C$ defined by $v_{ij}v_{(i+1)j}v_{(i+1)(j+1)}v_{i(j+1)}$, where $1\leq i\leq n-1$ and $0\leq j\leq m-2$, in a triangulation $T$ of a GCCG $G_{m,n}$ is of {\em type $A$} if its diagonal edge has no vertex in common with the diagonal edge of the cell defined by $v_{(i-1)j}v_{ij}v_{i(j+1)}v_{(i-1)(j+1)}$. We say that $C$ is of {\em type $B$} otherwise. Note that the type of cells on $L_1$ is not defined.

\begin{figure}[!htbp]
\begin{center}
\begin{tikzpicture}[scale=0.8,dot/.style={circle,inner sep=1pt,fill,name=#1}]
\def\dist{4}
\coordinate (BL) at (-1.5,-1.5) ; 
\coordinate (TR) at (1.5,1.5) ; 
\coordinate (Li) at (-2.2,-.5) ; 
\coordinate (bln) at (-1,-1);
\coordinate (O) at (0,0);

\begin{scope}[every coordinate/.style={shift={(0,0)}}]
\draw[step=1cm,very thin] ([c]BL) grid ([c]TR); 
\path ([c]Li) node[right] {$L_i$} ++(0,1) node[right] {$L_{i-1}$}
      ++(1.3,1.2) node[right] {$S_j$} ++(0.75,0) node[right] {$S_{j+1}$};
\fill[black!100] ([c]bln) circle(0.5ex) ++(0,1)circle(0.5ex) ++(0,1)circle(0.5ex)
	      ++(1,0)circle(0.5ex) ++(1,0)circle(0.5ex)
	      ++(0,-1)circle(0.5ex)  ++(0,-1)circle(0.5ex)
	      ++(-1,0)circle(0.5ex) ++(0,1)circle(0.5ex);
\draw[thick] ([c]bln)-- ++(0,1)-- ++(1,1)-- ++(1,0)-- ++(0,-2)-- ++(-2,0)-- ++(2,2)
	      ([c]bln)++(1,0)-- ++(0,2)
	      ([c]bln)++(0,1)-- ++(2,0)
	      ([c]bln)++(1,1)-- ++(1,-1);
\end{scope}

\begin{scope}[every coordinate/.style={shift={(\dist,0)}}]
\draw[step=1cm,very thin] ([c]BL) grid ([c]TR); 
\path ([c]Li) node[right] {$L_i$} ++(0,1) node[right] {$L_{i-1}$}
      ++(1.3,1.2) node[right] {$S_j$} ++(0.75,0) node[right] {$S_{j+1}$};
\fill[black!100] ([c]bln) circle(0.5ex) ++(0,1)circle(0.5ex) ++(0,1)circle(0.5ex)
	      ++(1,0)circle(0.5ex) ++(1,0)circle(0.5ex)
	      ++(0,-1)circle(0.5ex)  ++(0,-1)circle(0.5ex)
	      ++(-1,0)circle(0.5ex) ++(0,1)circle(0.5ex);
\draw[thick] ([c]O)-- ++(-1,-1)-- ++(1,0)-- ++(0,2)-- ++(-1,-1)-- ++(0,-1) ++(0,1)-- ++(2,0)
	       ([c]O)++(0,-1)-- ++(1,1)-- ++(-1,1);
\end{scope}

\begin{scope}[every coordinate/.style={shift={(2*\dist,0)}}]
\draw[step=1cm, thin] ([c]BL) grid ([c]TR); 
\path ([c]Li) node[right] {$L_i$} ++(0,1) node[right] {$L_{i-1}$}
      ++(1.3,1.2) node[right] {$S_j$} ++(0.75,0) node[right] {$S_{j+1}$};
\fill[black!100] ([c]bln) circle(0.5ex) ++(0,1)circle(0.5ex) ++(0,1)circle(0.5ex)
	      ++(1,0)circle(0.5ex) ++(1,0)circle(0.5ex)
	      ++(0,-1)circle(0.5ex)  ++(0,-1)circle(0.5ex)
	      ++(-1,0)circle(0.5ex) ++(0,1)circle(0.5ex);
\draw[thick] ([c]O)-- ++(1,1)-- ++(-2,0)-- ++(2,-2)-- ++(-1,0)-- ++(0,2) ++(1,0)-- ++(0,-2)
	     ++(0,1)-- ++(-2,0)++(0,1)--++(0,-1)--++(1,-1) ;%
\end{scope}
\begin{scope}[every coordinate/.style={shift={(3*\dist,0)}}]
\draw[step=1cm, thin] ([c]BL) grid ([c]TR); 
\path ([c]Li) node[right] {$L_i$} ++(0,1) node[right] {$L_{i-1}$}
      ++(1.3,1.2) node[right] {$S_j$} ++(0.75,0) node[right] {$S_{j+1}$};
\fill[black!100] ([c]bln) circle(0.5ex) ++(0,1)circle(0.5ex) ++(0,1)circle(0.5ex)
	      ++(1,0)circle(0.5ex) ++(1,0)circle(0.5ex)
	      ++(0,-1)circle(0.5ex)  ++(0,-1)circle(0.5ex)
	      ++(-1,0)circle(0.5ex) ++(0,1)circle(0.5ex);
\draw[thick] ([c]O)-- ++(-1,1)-- ++(0,-1)-- ++(2,0)-- ++(-1,1)-- ++(-1,0) ++(1,0)-- ++(0,-2)
	       ([c]O)++(-1,0)-- ++(1,-1)-- ++(1,1);	
\end{scope}
\end{tikzpicture}
\caption{The bottom-left cells are of type $A$ and the bottom-right cells are of type $B$}\label{just-4-cases}
\end{center}
\end{figure}
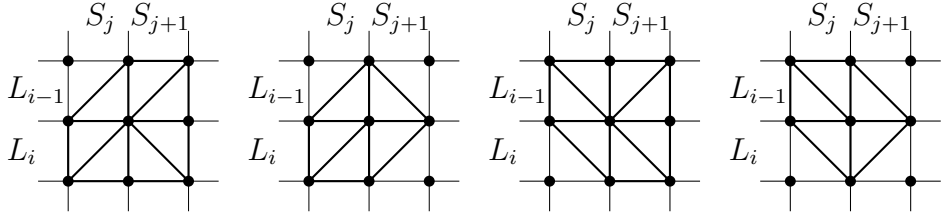

\begin{lem}\label{structure-lem} If a triangulation of a GCCG with more than three sectors contains no $W_5$ or $W_7$ as an induced subgraph then each cell on layer $L_i$, for $i\geq 2$, must be of the same type. \end{lem}

\begin{proof} If two cells on a layer $L_i$, $i\geq 2$, are of different types, then there must be two adjacent cells on $L_i$ of different types. Suppose that these cells are in the sectors $S_j$ and $S_{j+1}$. Considering these cells together with two cells in the same sectors on the layer $L_{i-1}$ we will meet either $W_5$ or $W_7$ as an induced subgraph, as shown in Figure~\ref{just-4-cases} (where we assumed that the bottom-left cells are of type $A$; the cases when these are of type $B$ are obtained from those in Figure~\ref{just-4-cases} by reflection with respect to a vertical line). Contradiction. \end{proof}

By Lemma~\ref{structure-lem}, each cell on a layer $L_i$, for $i\geq 2$, is of the same type, and thus the notion of a {\em layer type} (starting from layer 2 upwards) is well defined as the type of the layer's cells.

Next, we describe an orientation $O$ of a triangulation $T$ of a GCCG $G_{m,n}$ with $m\geq 4$, which will be shown by us in Lemma~\ref{orient-semi-trans-lem} to be semi-transitive.

\begin{itemize}
\item Orient the horizontal edges of $T$ as follows: for $0\leq x\leq n$ and $0\leq y\leq m-3$, $v_{x0}\rightarrow v_{x(m-1)}$, $v_{x(m-1)}\rightarrow v_{x(m-2)}$, and $v_{xy}\rightarrow v_{x(y+1)}$. Thus, all horizontal edges in the same sector receive the same orientation. In fact, we could pick any semi-transitive orientation of the cycle graph on $C_0$ and make the orientation of any other horizontal edge $h$ be the same as the orientation of the edge on $C_0$ belonging to $h$'s sector. However, we fixed a particular orientation, which is easy to deal with.
\item Each diagonal edge $d$ is oriented in the same direction as the horizontal edges of the cell $d$ belongs to. Thus, each horizontal or diagonal edge in a sector has the same orientation, which makes the {\em orientation of a sector} to be a well-defined notion.
\item Finally, orient vertical edges as follows: $v_{0y}\rightarrow v_{1y}$ for $0\leq y\leq m-1$. More generally, for $1\leq x\leq n-1$ and $0\leq y\leq m-1$, a vertical edge $v_{xy}v_{(x+1)y}$ has the same orientation as the edge  $v_{(x-1)y}v_{xy}$ if the layer $L_{x+1}$ is of type $A$, and it is oriented in the opposite direction if $L_{x+1}$ is of type $B$. Thus, we can orient all vertical edges, layer by layer, starting from layer $L_2$ and following our rules, so that all vertical edges on the same layer will be oriented in the same direction.
\end{itemize}

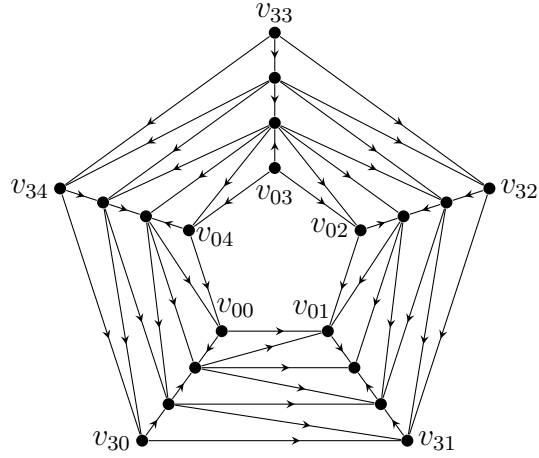
\begin{figure}[!htbp]
\begin{center}

\begin{tikzpicture}[scale=1.2,dot/.style={circle,inner sep=1.6pt,fill,name=#1}]
\path (-128:0.7 cm) node {$v_{00}$};
\path (-54:0.7 cm) node {$v_{01}$};
\path (29:0.7 cm) node {$v_{02}$};
\path (90:0.7 cm) node {$v_{03}$};
\path (160:0.7 cm) node {$v_{04}$};

\node[dot=L1] at (-0.588, -0.809) {} ;
\node[dot=L2] at (0.588, -0.809) {} ;
\node[dot=L3] at (0.951, 0.309) {} ;
\node[dot=L4] at (0., 1.) {} ;
\node[dot=L5] at (-0.951, 0.309) {} ;
\node[dot=L6] at (-0.882, -1.214) {} ;
\node[dot=L7] at (0.882, -1.214) {} ;
\node[dot=L8] at (1.427, 0.464) {} ;
\node[dot=L9] at (0., 1.5) {} ;
\node[dot=L10] at (-1.427, 0.464) {} ;
\node[dot=L11] at (-1.176, -1.618) {} ;
\node[dot=L12] at (1.176, -1.618) {} ;
\node[dot=L13] at (1.902, 0.618) {} ;
\node[dot=L14] at (0., 2.) {} ;
\node[dot=L15] at (-1.902, 0.618) {} ;
\node[dot=L16] at (-1.469, -2.023) {} ;
\node[dot=L17] at (1.469, -2.023) {} ;
\node[dot=L18] at (2.378, 0.773) {} ;
\node[dot=L19] at (0., 2.5) {} ;
\node[dot=L20] at (-2.378, 0.773) {} ;
\draw (L16)  node[left] {$v_{30}$};
\draw (L17)  node[right] {$v_{31}$};
\draw (L18)  node[right] {$v_{32}$};
\draw (L19)  node[above] {$v_{33}$};
\draw (L20)  node[left] {$v_{34}$};

\draw[,directed] (L1) -- (L2) ;
\draw[,directed] (L1) -- (L6) ;
\draw[,directed] (L2) -- (L7) ;
\draw[,directed] (L3) -- (L2) ;
\draw[,directed] (L3) -- (L8) ;
\draw[,directed] (L4) -- (L3) ;
\draw[,directed] (L4) -- (L5) ;
\draw[,directed] (L4) -- (L9) ;
\draw[,directed] (L5) -- (L1) ;
\draw[,directed] (L5) -- (L10) ;
\draw[,directed] (L6) -- (L2) ;
\draw[,directed] (L6) -- (L7) ;
\draw[,directed] (L6) -- (L12) ;
\draw[,directed] (L8) -- (L2) ;
\draw[,directed] (L8) -- (L7) ;
\draw[,directed] (L8) -- (L12) ;
\draw[,directed] (L9) -- (L3) ;
\draw[,directed] (L9) -- (L5) ;
\draw[,directed] (L9) -- (L8) ;
\draw[,directed] (L9) -- (L10) ;
\draw[,directed] (L9) -- (L13) ;
\draw[,directed] (L9) -- (L15) ;
\draw[,directed] (L10) -- (L1) ;
\draw[,directed] (L10) -- (L6) ;
\draw[,directed] (L10) -- (L11) ;
\draw[,directed] (L11) -- (L6) ;
\draw[,directed] (L11) -- (L12) ;
\draw[,directed] (L11) -- (L17) ;
\draw[,directed] (L12) -- (L7) ;
\draw[,directed] (L13) -- (L8) ;
\draw[,directed] (L13) -- (L12) ;
\draw[,directed] (L13) -- (L17) ;
\draw[,directed] (L14) -- (L9) ;
\draw[,directed] (L14) -- (L13) ;
\draw[,directed] (L14) -- (L15) ;
\draw[,directed] (L14) -- (L18) ;
\draw[,directed] (L14) -- (L20) ;
\draw[,directed] (L15) -- (L10) ;
\draw[,directed] (L15) -- (L11) ;
\draw[,directed] (L15) -- (L16) ;
\draw[,directed] (L16) -- (L11) ;
\draw[,directed] (L16) -- (L17) ;
\draw[,directed] (L17) -- (L12) ;
\draw[,directed] (L18) -- (L13) ;
\draw[,directed] (L18) -- (L17) ;
\draw[,directed] (L19) -- (L14) ;
\draw[,directed] (L19) -- (L18) ;
\draw[,directed] (L19) -- (L20) ;
\draw[,directed] (L20) -- (L15) ;
\draw[,directed] (L20) -- (L16) ;
\end{tikzpicture}
 \caption{The semi-transitive orientation $O$ on $T_{5,4}$}\label{ex-semi-tran-O}
\end{center}
\end{figure}

In what follows, when we refer to $C_0$, we mean the cycle graph induced by the vertices on $C_0$.

\begin{lem}\label{orient-semi-trans-lem} The orientation $O$ is semi-transitive. \end{lem}

\begin{proof}
First note that $O$ is acyclic. Indeed, any cycle must involve horizontal or diagonal edges, and since all such edges in a sector have the same direction, existence of a cycle in $O$ would force all sectors be oriented in the same direction contradicting the definition of $O$.

Suppose now that there is a shortcut edge $v_1\rightarrow v_k$, which is defined by a directed path $P=v_1\rightarrow v_2 \rightarrow \cdots \rightarrow v_k$, for $k\geq 4$. Taking into account that each edge in a sector has the same direction, no more than one (horizontal or diagonal) edge from each sector can be present in $P$ unless $P$ goes around the entire cylinder, which is not possible by the definition of $O$ (in particular, in such a situation $C_0$ would be forced to be a directed cycle).

Further, note that each directed path in $T$ induces a directed path on $C_0$ because each sector has the same orientation. In particular, steps on a vertical line correspond to no step on $C_0$. Taking into account this observation, there are only two possibilities for a shortcut:
\begin{itemize}

\item {\bf Case 1.} $P$ involves vertices from at least three consecutive different vertical lines, say $ V_x,  V_{x+1}$ and $V_{x+2}$ for some $x$, in order $P$ visits the lines; we also assume that $v_1$ is on $V_x$. However, once $P$ reaches $V_{x+2}$, the vertices on $V_{x+1}$ are not reachable for $P$ giving no possibility for a shortcut unless $P$ goes around the cylinder. But in the latter case, $C_0$ will be forced to have a shortcut or a cycle contradicting the definition of $O$.

\item {\bf Case 2.} Only two vertical lines are involved in $P$. By symmetry, only three subcases are possible for a  shortcut, which are shown in Figure~\ref{case2-figure}.

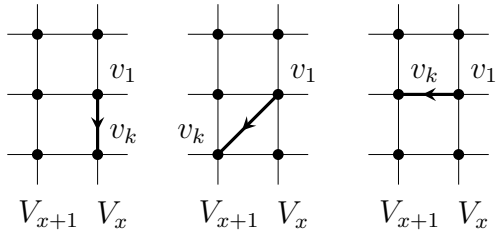
\begin{figure}[!htbp]
\begin{center}
\begin{tikzpicture}[scale=0.8]
\def\dist{3}
\coordinate (BL) at (-1.5,-1.5) ; 
\coordinate (TR) at (0.5,1.5) ; 
\coordinate (Vx) at (-1.5,-2) ; 
\coordinate (bln) at (-1,-1);
\coordinate (O) at (0,0);

\begin{scope}[every coordinate/.style={shift={(0,0)}}]
\draw[step=1cm,very thin] ([c]BL) grid ([c]TR); 
\path ([c]Vx) node[right] {$V_{x+1}$} ++(1.3,0) node[right] {$V_{x}$};

\foreach \x in {-1,...,0} 
 \foreach \y in {-1,...,1}
 { \fill[black!100] [c](\x,\y) circle(0.5ex);
 }
 \draw[very thick,directed] [c](0,0) node[above right] {$v_1$}-- ++(0,-1) node[above right] {$v_k$};

\end{scope}

\begin{scope}[every coordinate/.style={shift={(\dist,0)}}]

\draw[step=1cm,very thin] ([c]BL) grid ([c]TR); 
\path ([c]Vx) node[right] {$V_{x+1}$} ++(1.3,0) node[right] {$V_{x}$};

\foreach \x in {-1,...,0} 
 \foreach \y in {-1,...,1}
 { \fill[black!100] [c](\x,\y) circle(0.5ex);
 }
 \draw[very thick,directed] [c](0,0) node[above right] {$v_1$}-- ++(-1,-1) node[above left] {$v_k$};
\end{scope}

\begin{scope}[every coordinate/.style={shift={(2*\dist,0)}}]

\draw[step=1cm,very thin] ([c]BL) grid ([c]TR); 
\path ([c]Vx) node[right] {$V_{x+1}$} ++(1.3,0) node[right] {$V_{x}$};

\foreach \x in {-1,...,0} 
 \foreach \y in {-1,...,1}
 { \fill[black!100] [c](\x,\y) circle(0.5ex);
 }
 \draw[very thick,directed] [c](0,0) node[above right] {$v_1$}-- ++(-1,0) node[above right] {$v_k$};

\end{scope}
\end{tikzpicture}
\caption{\label{case2-figure} Three possibilities for Case 2}
\end{center}
\end{figure}

 {\bf Subcase 2.1.} Consider the leftmost picture in Figure \ref{case2-figure}.
 The orientations of edges between vertical lines $V_x$ and $V_{x+1}$ ensure that $v_2$ cannot be on $V_{x+1}$. But then the presence of the edge $v_1\rightarrow v_2$ shows that the cell containing $\{v_1,v_k,y\}$ in  Figure~\ref{subcase2.1} is of type $B$, so that only two situations are possible here, both presented in  Figure~\ref{subcase2.1}. Note that in any case,  $v_2\rightarrow v_k$ cannot be an edge. But then, because of the the orientation of the sector defined by the lines $V_x$ and $V_{x+1}$, the path $P$ can never reach $v_k$. Contradiction.

 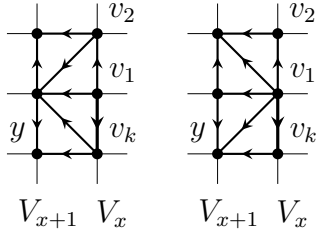
\begin{figure}[!htbp]
 \begin{center}
\begin{tikzpicture}[scale=0.8]
\def\dist{3}
\coordinate (BL) at (-1.5,-1.5) ; 
\coordinate (TR) at (0.5,1.5) ; 
\coordinate (Vx) at (-1.5,-2) ; 

\begin{scope}[every coordinate/.style={shift={(0,0)}}]
\draw[step=1cm,very thin] ([c]BL) grid ([c]TR); 
\path ([c]Vx) node[right] {$V_{x+1}$} ++(1.3,0) node[right] {$V_{x}$};

\foreach \x in {-1,...,0} 
 \foreach \y in {-1,...,1}
 { \fill[black!100] [c](\x,\y) circle(0.5ex);
 }
 \draw[very thick,directed] [c](0,0) node[above right] {$v_1$}-- ++(0,-1) node[above right] {$v_k$};
\draw[thick,directed] [c](0,0)-- ++(-1,0);
\draw[thick,directed] [c](0,0)-- ++(0,1);

\draw[thick,directed] [c](0,1) node[above right] {$v_2$}-- ++(-1,0);
\draw[thick,directed] [c](0,1)-- ++(-1,-1);

\draw[thick,directed] [c](0,-1)-- ++(-1,0);
\draw[thick,directed] [c](0,-1)-- ++(-1,1);

\draw[thick,directed] [c](-1,0)-- ++(0,1);
\draw[thick,directed] [c](-1,0)-- ++(0,-1)node[above left] {$y$};
\end{scope}
\begin{scope}[every coordinate/.style={shift={(\dist,0)}}]
\draw[step=1cm,very thin] ([c]BL) grid ([c]TR); 
\path ([c]Vx) node[right] {$V_{x+1}$} ++(1.3,0) node[right] {$V_{x}$};

\foreach \x in {-1,...,0} 
 \foreach \y in {-1,...,1}
 { \fill[black!100] [c](\x,\y) circle(0.5ex);
 }
 \draw[very thick,directed] [c](0,0) node[above right] {$v_1$}-- ++(0,-1) node[above right] {$v_k$};
\draw[thick,directed] [c](0,0)-- ++(-1,0);
\draw[thick,directed] [c](0,0)-- ++(0,1);
\draw[thick,directed] [c](0,0)-- ++(-1,1);
\draw[thick,directed] [c](0,0)-- ++(-1,-1);
\draw[thick,directed] [c](0,1) node[above right] {$v_2$}-- ++(-1,0);
\draw[thick,directed] [c](0,-1)-- ++(-1,0);
\draw[thick,directed] [c](-1,0)-- ++(0,1);
\draw[thick,directed] [c](-1,0)-- ++(0,-1)node[above left] {$y$};
\end{scope}
\end{tikzpicture}
\caption{\label{subcase2.1} Possibilities for Subcase 2.1}
\end{center}
\end{figure}

 {\bf{Subcase 2.2}}. Note that $v_1\rightarrow v_2$ cannot be a horizontal edge. Since $k\geq 4$, we consider four possible situations presented in Figure~\ref{subcase2.2}.
  \begin{itemize}
  \item[(1)] Suppose that $v_1\rightarrow v_2$ is as shown in the rightmost picture in Figure~\ref{subcase2.2}. Since the cell defined by $\{v_1,a,v_k,c\}$ is of type $B$, the edges $av_2$ and $av_k$ will receive opposite directions, so that $v_1\rightarrow v_k$ cannot be a shortcut in this situation.

  \item[(2)] If $v_1\rightarrow v_2$ is an edge as shown in the second picture in Figure~\ref{subcase2.2}, then we have also an edge $a\rightarrow b$, and because of that, $v_2\rightarrow a$ must also be an edge (otherwise, $P$ has no possibility to reach eventually $v_k$). However, if $a\rightarrow b$ is an edge, the cell defined by $\{c,v_1,a,v_k\}$ is of type $A$, and thus $ v_k\rightarrow a$ is an edge, and $P$ cannot be a shortcut (the vertices $v_1$, $v_2$, $a$ and $v_k$ do not define a shortcut).

  \item[(3)] Suppose that $v_1\rightarrow v_2$ and $v_2\rightarrow c$ are edges as shown in the third picture in Figure~\ref{subcase2.2}. Then the cell defined by $\{v_2,v_k,b,c\}$ is of type $A$ (which is reflected in Figure~\ref{subcase2.2}) and clearly $P$ will not reach $v_k$. Contradiction.

  \item[(4)] Finally, suppose that $v_1\rightarrow v_2$ and $c\rightarrow v_2$ are edges as shown in the fourth picture in Figure~\ref{subcase2.2}. Then the cell defined by $\{v_2,v_k,b,c\}$ is of type $B$ (which is reflected in Figure~\ref{subcase2.2}) and clearly $P$ will not reach $v_k$. Contradiction.

  \end{itemize}
  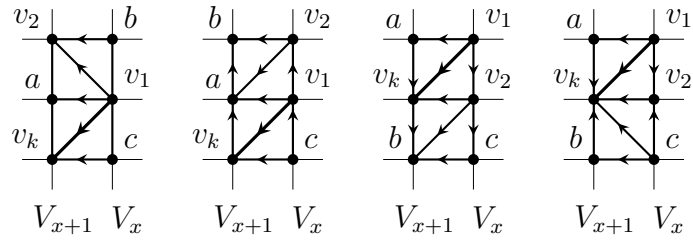
\begin{figure}[!htbp]
\begin{center}
\begin{tikzpicture}[scale=0.8]
\def\dist{3}
\coordinate (BL) at (-1.5,-1.5) ; 
\coordinate (TR) at (0.5,1.5) ; 
\coordinate (Vx) at (-1.5,-2) ; 
\coordinate (bln) at (-1,-1);
\coordinate (O) at (0,0);

\begin{scope}[every coordinate/.style={shift={(-1*\dist,0)}}]
\draw[step=1cm,very thin] ([c]BL) grid ([c]TR); 
\path ([c]Vx) node[right] {$V_{x+1}$} ++(1.3,0) node[right] {$V_{x}$};

\foreach \x in {-1,...,0} 
 \foreach \y in {-1,...,1}
 { \fill[black!100] [c](\x,\y) circle(0.5ex);
 }
 \draw[very thick,directed] [c](0,0) node[above right] {$v_1$}-- ++(-1,-1) node[above left] {$v_k$};
\draw[thick] [c](0,0) -- ++(0,-1)node[above right] {$c$};
\draw[thick,directed] [c](0,0)-- ++(-1,0);
\draw[thick,directed] [c](0,0)-- ++(-1,1);
\draw[thick] [c](0,0)-- ++(0,1);
\draw[thick,directed] [c](0,1) node[above right] {$b$}-- ++(-1,0) node[above left] {$v_2$};
\draw[thick,directed] [c](0,-1)-- ++(-1,0);
\draw[thick] [c](-1,0)node[above left] {$a$}-- ++(0,1);
\draw[thick] [c](-1,0)-- ++(0,-1);
\end{scope}

\begin{scope}[every coordinate/.style={shift={(0,0)}}]
\draw[step=1cm,very thin] ([c]BL) grid ([c]TR); 
\path ([c]Vx) node[right] {$V_{x+1}$} ++(1.3,0) node[right] {$V_{x}$};

\foreach \x in {-1,...,0} 
 \foreach \y in {-1,...,1}
 { \fill[black!100] [c](\x,\y) circle(0.5ex);
 }
 \draw[very thick,directed] [c](0,0) node[above right] {$v_1$}-- ++(-1,-1) node[above left] {$v_k$};
\draw[thick,reverse directed] [c](0,0) -- ++(0,-1)node[above right] {$c$};
\draw[thick,directed] [c](0,0)-- ++(-1,0);
\draw[thick,directed] [c](0,0)-- ++(0,1);
\draw[thick,directed] [c](0,1) node[above right] {$v_2$}-- ++(-1,0);
\draw[thick,directed] [c](0,1)-- ++(-1,-1);
\draw[thick,directed] [c](0,-1)-- ++(-1,0);
\draw[thick,directed] [c](-1,0)node[above left] {$a$}-- ++(0,1)node[above left] {$b$};
\draw[thick,reverse directed] [c](-1,0)-- ++(0,-1);
\end{scope}
\begin{scope}[every coordinate/.style={shift={(\dist,0)}}]
\draw[step=1cm,very thin] ([c]BL) grid ([c]TR); 
\path ([c]Vx) node[right] {$V_{x+1}$} ++(1.3,0) node[right] {$V_{x}$};
\foreach \x in {-1,...,0} 
 \foreach \y in {-1,...,1}
 { \fill[black!100] [c](\x,\y) circle(0.5ex);
 }
\draw[very thick,directed] [c](0,1) node[above right] {$v_1$}-- ++(-1,-1) node[above left] {$v_k$};
\draw[thick,directed] [c](0,1)-- ++(-1,0)node[above left] {$a$};
\draw[thick,directed] [c](0,1)-- ++(0,-1);
\draw[thick,directed][c](0,0) node[above right] {$v_2$}-- ++(0,-1);
\draw[thick,directed] [c](0,0)-- ++(-1,0);
\draw[thick,directed] [c](0,0)-- ++(-1,-1);
\draw[thick,directed] [c](0,-1)node[above right] {$c$}-- ++(-1,0);
\draw[thick,reverse directed] [c](-1,0)-- ++(0,1);
\draw[thick,directed] [c](-1,0)-- ++(0,-1)node[above left] {$b$};
\end{scope}
\begin{scope}[every coordinate/.style={shift={(2*\dist,0)}}]
\draw[step=1cm,very thin] ([c]BL) grid ([c]TR); 
\path ([c]Vx) node[right] {$V_{x+1}$} ++(1.3,0) node[right] {$V_{x}$};
\foreach \x in {-1,...,0} 
 \foreach \y in {-1,...,1}
 { \fill[black!100] [c](\x,\y) circle(0.5ex);
 }
\draw[very thick,directed] [c](0,1) node[above right] {$v_1$}-- ++(-1,-1) node[above left] {$v_k$};
\draw[thick,directed] [c](0,1)-- ++(-1,0)node[above left] {$a$};
\draw[thick,directed] [c](0,1)-- ++(0,-1);
 \draw[ thick,reverse directed] [c](0,0) node[above right] {$v_2$}-- ++(0,-1);
\draw[thick, directed] [c](0,0)-- ++(-1,0);
\draw[thick,directed] [c](0,-1)node[above right] {$c$}-- ++(-1,0);
\draw[thick,directed] [c](0,-1)-- ++(-1,1);
\draw[thick,reverse directed] [c](-1,0)-- ++(0,1);
\draw[thick,reverse directed] [c](-1,0)-- ++(0,-1)node[above left] {$b$};
\end{scope}
\end{tikzpicture}
\caption{\label{subcase2.2} Possibilities for Subcase 2.2}
\end{center}
\end{figure}

 {\bf Subcase 2.3.} For the rightmost picture in Figure~\ref{case2-figure}, we omit our arguments since they are very  similar to the arguments in Subcases 2.1 and 2.2.
\end{itemize}
The  lemma is proved.
\end{proof}

\section{Word-representable triangulations of GCCGs with three sectors}\label{sec4}

The goal of this section is to prove the following theorem.

\begin{thm}\label{thm-3} A triangulation of a GCCG with three sectors is word-repre-sentable if and only if it contains no graph in Figure~{\em \ref{non-representabiliy-of-6-graphs}} as an induced subgraph.\end{thm}

Our proof is organized as follows. In Subsection~\ref{non-representabiliy-of-6-graphs} we will provide all six minimum non-word-representable graphs that can appear in triangulations of GCCGs with three sectors (see Figure~\ref{non-representabiliy-of-6-graphs}) and give an explicit proof that one of these graphs is non-word-representable. Then, in Subsection~\ref{inductive-arg-subsec}, we will give an inductive argument showing that avoidance of the six graphs in Figure~\ref{non-repr-induced-subgraphs} is a sufficient condition for a GCCG with three sectors to be word-representable.  Note that the graphs in Figure~\ref{non-repr-induced-subgraphs}  were obtained by an exhaustive computer search on graphs on up to eight vertices. However, our argument in Subsection~\ref{inductive-arg-subsec} will show that no other non-word-representable induced subgraphs can be found among all triangulations of GCCGs with three sectors.

\subsection{Non-word-representability of the graphs in Figure~\ref{non-repr-induced-subgraphs}}\label{non-representabiliy-of-6-graphs}

Non-word-representability of the graphs in Figure~\ref{non-repr-induced-subgraphs} can be checked using existing software~\cite{Glen}. However, there is a way to check this fact by hand using the branching approach, which is rather space consuming, and thus we will demonstrate this approach only on one example, the second graph in Figure~\ref{non-repr-induced-subgraphs}; the remaining cases can be checked similarly.

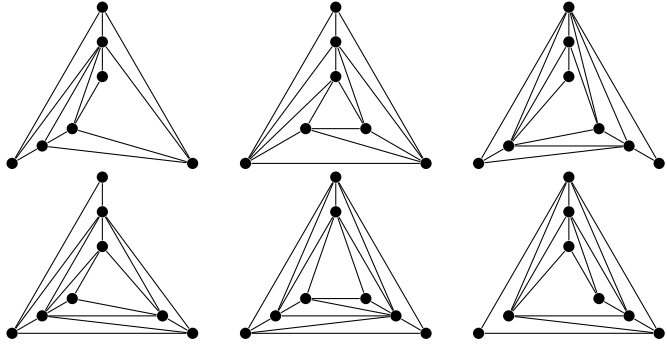
\begin{figure}[!htbp]
\begin{center}

\begin{tikzpicture}[scale=0.8,dot/.style={circle,inner sep=1.5pt,fill,name=#1}]
\node[dot=L10] at (-0.5, -0.289) {} ;
\node[dot=L12] at (0, 0.577) {} ;
\node[dot=L20] at (-1, -0.577) {} ;
\node[dot=L22] at (0, 1.154) {} ;
\node[dot=L30] at (-1.5, -0.867) {} ;
\node[dot=L21] at (1.5, -0.867) {} ;
\node[dot=L32] at (0, 1.731) {} ;
\draw[] (L12)-- (L10) ;
\draw[] (L10)-- (L20) ;
\draw[] (L20)-- (L21) ;
\draw[] (L10)-- (L21) ;
\draw[] (L21)-- (L22) ;
\draw[] (L22)-- (L20) ;
\draw[] (L12)-- (L22) ;
\draw[] (L10)-- (L22) ;
\draw[] (L20)-- (L30) ;
\draw[] (L22)-- (L30) ;
\draw[] (L32)-- (L30) ;
\draw[] (L22)-- (L32) ;
\draw[] (L21)-- (L32) ;
\end{tikzpicture}
\quad
\begin{tikzpicture}[scale=0.8,dot/.style={circle,inner sep=1.5pt,fill,name=#1}]
\node[dot=L10] at (-0.5, -0.289) {} ;
\node[dot=L11] at (0.5, -0.289) {} ;
\node[dot=L12] at (0, 0.577) {} ;
\node[dot=L22] at (0, 1.154) {} ;
\node[dot=L20] at (-1.5, -0.867) {} ;
\node[dot=L21] at (1.5, -0.867) {} ;
\node[dot=L32] at (0, 1.731) {} ;
\draw[] (L10)-- (L11) ;
\draw[] (L11)-- (L12) ;
\draw[] (L12)-- (L10) ;
\draw[] (L10)-- (L20) ;
\draw[] (L12)-- (L20) ;
\draw[] (L20)-- (L21) ;
\draw[] (L11)-- (L21) ;
\draw[] (L10)-- (L21) ;
\draw[] (L21)-- (L22) ;
\draw[] (L22)-- (L20) ;
\draw[] (L12)-- (L22) ;
\draw[] (L11)-- (L22) ;
\draw[] (L22)-- (L32) ;
\draw[] (L21)-- (L32) ;
\draw[] (L20)-- (L32) ;
\end{tikzpicture}
\quad
\begin{tikzpicture}[scale=0.8,dot/.style={circle,inner sep=1.5pt,fill,name=#1}]
\node[dot=L21] at (0.5, -0.289) {} ;
\node[dot=L12] at (0, 0.577) {} ;
\node[dot=L20] at (-1, -0.577) {} ;
\node[dot=L31] at (1, -0.577) {} ;
\node[dot=L22] at (0, 1.154) {} ;
\node[dot=L30] at (-1.5, -0.867) {} ;
\node[dot=L41] at (1.5, -0.867) {} ;
\node[dot=L32] at (0, 1.731) {} ;
\draw[] (L12)-- (L20) ;
\draw[] (L20)-- (L21) ;
\draw[] (L21)-- (L22) ;
\draw[] (L22)-- (L20) ;
\draw[] (L12)-- (L22) ;
\draw[] (L20)-- (L30) ;
\draw[] (L30)-- (L31) ;
\draw[] (L21)-- (L31) ;
\draw[] (L20)-- (L31) ;
\draw[] (L31)-- (L32) ;
\draw[] (L32)-- (L30) ;
\draw[] (L22)-- (L32) ;
\draw[] (L21)-- (L32) ;
\draw[] (L20)-- (L32) ;
\draw[] (L31)-- (L41) ;
\draw[] (L32)-- (L41) ;
\end{tikzpicture}

\begin{tikzpicture}[scale=0.8,dot/.style={circle,inner sep=1.5pt,fill,name=#1}]
\node[dot=L10] at (-0.5, -0.289) {} ;
\node[dot=L22] at (0, 0.577) {} ;
\node[dot=L20] at (-1, -0.577) {} ;
\node[dot=L21] at (1, -0.577) {} ;
\node[dot=L32] at (0, 1.154) {} ;
\node[dot=L30] at (-1.5, -0.867) {} ;
\node[dot=L31] at (1.5, -0.867) {} ;
\node[dot=L42] at (0, 1.731) {} ;
\draw[] (L10)-- (L20) ;
\draw[] (L20)-- (L21) ;
\draw[] (L10)-- (L21) ;
\draw[] (L21)-- (L22) ;
\draw[] (L22)-- (L20) ;
\draw[] (L10)-- (L22) ;
\draw[] (L20)-- (L30) ;
\draw[] (L30)-- (L31) ;
\draw[] (L21)-- (L31) ;
\draw[] (L20)-- (L31) ;
\draw[] (L31)-- (L32) ;
\draw[] (L32)-- (L30) ;
\draw[] (L22)-- (L32) ;
\draw[] (L21)-- (L32) ;
\draw[] (L20)-- (L32) ;
\draw[] (L32)-- (L42) ;
\draw[] (L30)-- (L42) ;
\end{tikzpicture}
\quad
\begin{tikzpicture}[scale=0.8,dot/.style={circle,inner sep=1.5pt,fill,name=#1}]
\node[dot=L10] at (-0.5, -0.289) {} ;
\node[dot=L11] at (0.5, -0.289) {} ;
\node[dot=L20] at (-1, -0.577) {} ;
\node[dot=L21] at (1, -0.577) {} ;
\node[dot=L22] at (0, 1.154) {} ;
\node[dot=L30] at (-1.5, -0.867) {} ;
\node[dot=L31] at (1.5, -0.867) {} ;
\node[dot=L32] at (0, 1.731) {} ;
\draw[] (L10)-- (L11) ;
\draw[] (L10)-- (L20) ;
\draw[] (L20)-- (L21) ;
\draw[] (L11)-- (L21) ;
\draw[] (L10)-- (L21) ;
\draw[] (L21)-- (L22) ;
\draw[] (L22)-- (L20) ;
\draw[] (L11)-- (L22) ;
\draw[] (L10)-- (L22) ;
\draw[] (L20)-- (L30) ;
\draw[] (L21)-- (L30) ;
\draw[] (L30)-- (L31) ;
\draw[] (L21)-- (L31) ;
\draw[] (L31)-- (L32) ;
\draw[] (L32)-- (L30) ;
\draw[] (L22)-- (L32) ;
\draw[] (L21)-- (L32) ;
\draw[] (L20)-- (L32) ;
\end{tikzpicture}
\quad
\begin{tikzpicture}[scale=0.8,dot/.style={circle,inner sep=1.5pt,fill,name=#1}]
\node[dot=L11] at (0.5, -0.289) {} ;
\node[dot=L12] at (0, 0.577) {} ;
\node[dot=L20] at (-1, -0.577) {} ;
\node[dot=L21] at (1, -0.577) {} ;
\node[dot=L22] at (0, 1.154) {} ;
\node[dot=L40] at (-1.5, -0.867) {} ;
\node[dot=L31] at (1.5, -0.867) {} ;
\node[dot=L32] at (0, 1.731) {} ;
\draw[] (L11)-- (L12) ;
\draw[] (L12)-- (L20) ;
\draw[] (L20)-- (L21) ;
\draw[] (L11)-- (L21) ;
\draw[] (L21)-- (L22) ;
\draw[] (L22)-- (L20) ;
\draw[] (L12)-- (L22) ;
\draw[] (L11)-- (L22) ;
\draw[] (L21)-- (L31) ;
\draw[] (L20)-- (L31) ;
\draw[] (L31)-- (L32) ;
\draw[] (L22)-- (L32) ;
\draw[] (L21)-- (L32) ;
\draw[] (L20)-- (L32) ;
\draw[] (L31)-- (L40) ;
\draw[] (L32)-- (L40) ;

\end{tikzpicture}
\caption{All minimal non-word-representable induced sugraphs in triangulations of GCCG's with three sectors}\label{non-repr-induced-subgraphs}
\end{center}
\end{figure}

Note that for any of the partial orientations of the 3- or 4-cycles
given in Figure~\ref{rules}, there is a unique way of completing
these orientations, also shown in  Figure~\ref{rules}, so that oriented cycles and shortcuts are avoided. This stays true in the context of triangulated 4-cycles, because such graphs are different from $K_{4}$ admitting an alternative semi-transitive (in fact, transitive) orientation completion.

\begin{figure}[h]
\begin{center}
\setlength{\unitlength}{3mm}
\begin{picture}(5,4.5)
\put(-16,0){


%

\put(0,0){\p}
\put(4,0){\p}
\put(4,3){\p}

\put(0,0){\vector(1,0){3.9}}
\put(4,0){\vector(0,1){2.9}}
\put(0,0){\line(4,3){3.9}}
\put(5,1){$\longmapsto$}

}
\put(-9,0){

\put(0,0){\p}
\put(4,0){\p}
\put(4,3){\p}

\put(0,0){\vector(1,0){3.9}}
\put(4,0){\vector(0,1){2.9}}
\put(0,0){\vector(4,3){3.9}}

}

\put(-3.5,0){

\put(4,0){\p}
\put(4,3){\p}
\put(0,3){\p}
\put(0,0){\p}

\put(0,0){\vector(1,0){3.9}}
\put(4,0){\vector(0,1){2.9}}
\put(0,0){\line(0,1){2.9}}
\put(0,3){\line(1,0){3.9}}

\put(5,1){$\longmapsto$}
}

\put(4,0){


\put(4,0){\p}
\put(4,3){\p}
\put(0,3){\p}
\put(0,0){\p}

\put(0,0){\vector(1,0){3.9}}
\put(4,0){\vector(0,1){2.9}}
\put(0,0){\vector(0,1){2.9}}
\put(0,3){\vector(1,0){3.9}}

}

\put(9.5,0){

\put(4,0){\p}
\put(4,3){\p}
\put(0,3){\p}
\put(0,0){\p}

\put(0,0){\vector(1,0){3.9}}
\put(4,0){\line(0,1){2.9}}
\put(0,0){\line(0,1){2.9}}
\put(4,3){\vector(-1,0){3.9}}

\put(5,1){$\longmapsto$}
}

\put(17,0){

\put(4,0){\p}
\put(4,3){\p}
\put(0,3){\p}
\put(0,0){\p}

\put(0,0){\vector(1,0){3.9}}
\put(4,3){\vector(0,-1){2.9}}
\put(0,0){\vector(0,1){2.9}}
\put(4,3){\vector(-1,0){3.9}}

}

\end{picture}


\caption{Unique way of completing in a semi-transitive\index{semi-transitive orientation} way partial orientations of a 3-cycle or a 4-cycle}
\label{rules}
\end{center}
\end{figure}
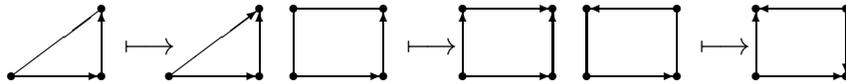

Below,  we use the following terminology introduced in~\cite{Akrobotu}. \emph{Complete XYW(Z)} refers to completing the orientations on a cycle $XYW(Z)$ according to the respective cases in Figure \ref{rules}. Instances in which it is not possible
to uniquely determine orientations of any additional edges in a partially\index{partially oriented graph} oriented graph
are referred to as \emph{Branching XY}. Here, one picks a new, still non-oriented edge
$XY$ of the graph and assigns the orientation $X\rightarrow Y$, while, at the
same time, one makes a copy of the graph, respectively, with its partial orientations
and assigns orientation $Y\rightarrow X$ to the edge $XY$. The new copy is named and examined
later on. Our terminology and relevant abbreviations are
summarized in Table~\ref{table1}.

\begin{table}[h!]
\centering
\begin{tabular}{|c|l|}
\hline
{\bf Abbreviation} & {\bf Operation} \\
\hline
B & Branch \\
NC & Obtain a new partially oriented copy \\
C & Complete \\
MC & Move to a copy \\
S & Obtain a shortcut\index{shortcut} \\

\hline
\end{tabular}
\caption{List of used operations and their abbreviations}
\label{table1}
\end{table}

Name $A$ the first copy of the second graph in Figure~\ref{non-repr-induced-subgraphs} with the single edge orientated as $1\rightarrow 3$, and carry out the following operations, where the  partially oriented graphs $A$ -- $L$ are given in Figure~\ref{12-partially-orient-graphs}. We will show that in case of any acyclic orientation of the graph, a shortcut is inventible.

\begin{itemize}

\item B 37 (NC $B$), C 137, C 1376, B 65 (NC $C$), C 1654, C 1764, C 4576, C 3752, C 2564, C 2457, C 123, S 1324;

\item MC $C$, C 567, C 5672, B 45 (NC $D$), C 1654, C456, C 1452, C 452, C 2573, S 4132;

\item MC $D$, C 1654, C 1452, C 2754, C 2413, S 5237;

\item MC $B$, B 16 (NC $E$), C 6137, B 17 (NC $F$), C 1764, B 45 (NC $G$), C 1452, C 6145, C 765, C 7652, C 2573, C 2564, S 1423;

\item MC $G$, C 546, C 5417, C 6457, C 5732, C 1452, C 2564, S 1423;

\item MC $F$, C 7165, B 25 (NC $H$), C 2567, C 273, C 2754, C 1654, C 2541, C 456, S 4271;

\item MC $H$, C 2573, C 132, C 732, C 1324, C 4125, C 4576, S 1642;

\item MC $E$, C 7316, B 12 (NC $I$), C 6125, C 7652, C 6524, B 17 (NC $J$), C 1764, C 4125, C 5467, C 2573, S 1423;

\item MC $J$, C 1764, C 4125, C 5467, C 3752, S 4132;

\item MC $I$, C 213, B 14 (NC $K$), C 214, C 614, C 6145, C 7652, C 7325, C 7561, S 2714;

\item MC $K$, C 4132, B 25 (NC $L$), C 5214, C 6145, C 7652, C 1427, C 1754, S 2573;

\item MC $L$, C 4125, C 6145, C 7652, C 5237, C 7541, C 4576, S 6427.

\end{itemize}

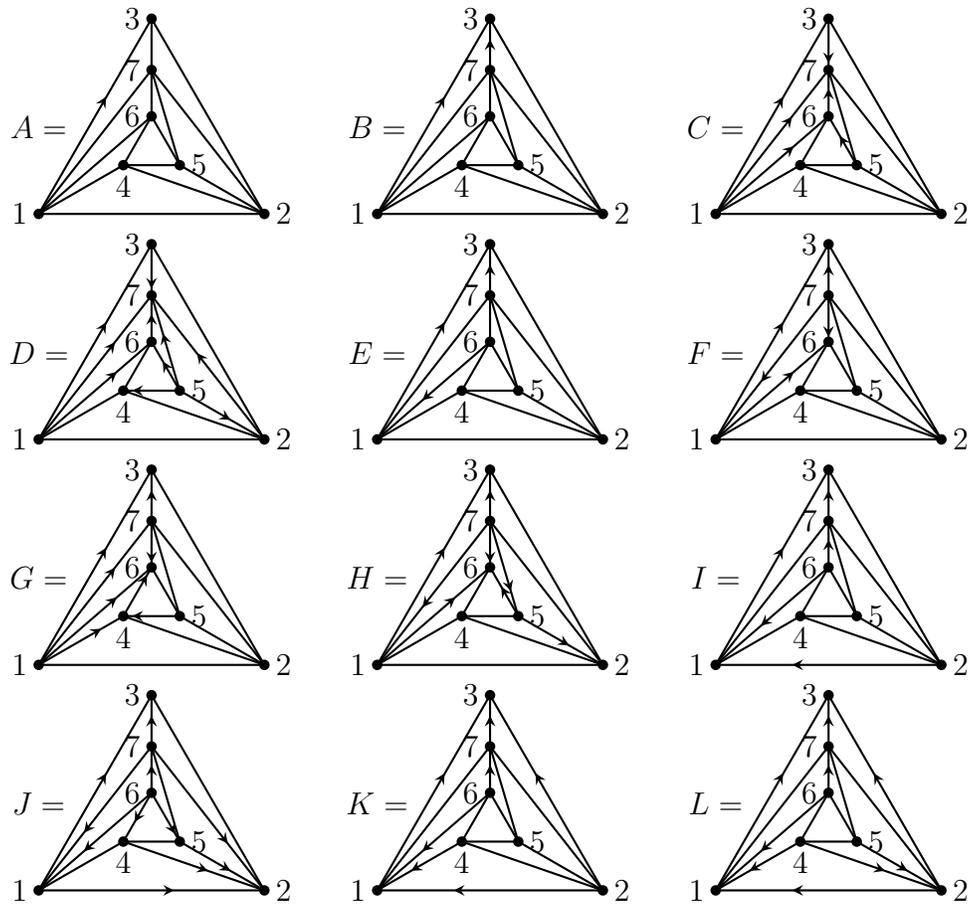
\begin{figure}[!htbp]
\begin{center}
\begin{tikzpicture}[scale=1.5,dot/.style={circle,inner sep=1.4pt,fill}]
\coordinate (a) at (-1,0);
\coordinate (L10) at (-0.5/2, -0.289/2) ;
\coordinate (L11) at (0.5/2, -0.289/2);
\coordinate (L12) at (0, 0.577/2);
\coordinate (L20) at (-1, -0.577);
\coordinate (L21) at (1, -0.577);
\coordinate (L22) at (0, 0.7);
\coordinate (L32) at (0, 1.154);

\begin{scope}[every coordinate/.style={shift={(0,0)}}]
\node[above] at ([c]a) {$A=$};
\draw[thick] ([c]L10) node[dot]{}-- ([c]L11) ;
\draw[thick] ([c]L11)node[dot]{}-- ([c]L12) ;
\draw[thick] ([c]L12)node[dot]{}-- ([c]L10) ;
\draw[thick] ([c]L10)node[below]{4}-- ([c]L20) ;
\draw[thick] ([c]L12)node[left]{6}-- ([c]L20) ;
\draw[thick] ([c]L20)node[dot]{}-- ([c]L21) ;
\draw[thick] ([c]L11)node[right]{5}-- ([c]L21) ;
\draw[thick] ([c]L10)-- ([c]L21) ;
\draw[thick] ([c]L21)node[dot]{}-- ([c]L22) ;
\draw[thick] ([c]L22)node[dot]{}-- ([c]L20) ;
\draw[thick] ([c]L12)-- ([c]L22) ;
\draw[thick] ([c]L11)-- ([c]L22) ;
\draw[thick] ([c]L22)node[left]{7}-- ([c]L32)node[dot]{} ;
\draw[thick] ([c]L21)node[right]{2}-- ([c]L32)node[left]{3} ;
\draw[thick,directed] ([c]L20)node[left]{1}-- ([c]L32) ;
\end{scope}

\begin{scope}[every coordinate/.style={shift={(3,0)}}]
\node[above] at ([c]a) {$B=$};
\draw[thick] ([c]L10) node[dot]{}-- ([c]L11) ;
\draw[thick] ([c]L11)node[dot]{}-- ([c]L12) ;
\draw[thick] ([c]L12)node[dot]{}-- ([c]L10) ;
\draw[thick] ([c]L10)node[below]{4}-- ([c]L20) ;
\draw[thick] ([c]L12)node[left]{6}-- ([c]L20) ;
\draw[thick] ([c]L20)node[dot]{}-- ([c]L21) ;
\draw[thick] ([c]L11)node[right]{5}-- ([c]L21) ;
\draw[thick] ([c]L10)-- ([c]L21) ;
\draw[thick] ([c]L21)node[dot]{}-- ([c]L22) ;
\draw[thick] ([c]L22)node[dot]{}-- ([c]L20) ;
\draw[thick] ([c]L12)-- ([c]L22) ;
\draw[thick] ([c]L11)-- ([c]L22) ;
\draw[thick,directed] ([c]L22)node[left]{7}-- ([c]L32)node[dot]{} ;
\draw[thick] ([c]L21)node[right]{2}-- ([c]L32)node[left]{3} ;
\draw[thick,directed] ([c]L20)node[left]{1}-- ([c]L32) ;
\end{scope}

\begin{scope}[every coordinate/.style={shift={(6,0)}}]
\node[above] at ([c]a) {$C=$};
\draw[thick] ([c]L10) node[dot]{}-- ([c]L11) ;
\draw[thick, directed] ([c]L11)node[dot]{}-- ([c]L12) ;
\draw[thick] ([c]L12)node[dot]{}-- ([c]L10) ;
\draw[thick] ([c]L10)node[below]{4}-- ([c]L20) ;
\draw[thick,reverse directed] ([c]L12)node[left]{6}-- ([c]L20) ;
\draw[thick] ([c]L20)node[dot]{}-- ([c]L21) ;
\draw[thick] ([c]L11)node[right]{5}-- ([c]L21) ;
\draw[thick] ([c]L10)-- ([c]L21) ;
\draw[thick] ([c]L21)node[dot]{}-- ([c]L22) ;
\draw[thick,reverse directed] ([c]L22)node[dot]{}-- ([c]L20) ;
\draw[thick, directed] ([c]L12)-- ([c]L22) ;
\draw[thick] ([c]L11)-- ([c]L22) ;
\draw[thick,reverse directed] ([c]L22)node[left]{7}-- ([c]L32)node[dot]{} ;
\draw[thick] ([c]L21)node[right]{2}-- ([c]L32)node[left]{3} ;
\draw[thick,directed] ([c]L20)node[left]{1}-- ([c]L32) ;
\end{scope}

\begin{scope}[every coordinate/.style={shift={(0,-2)}}]
\node[above] at ([c]a) {$D=$};
\draw[thick,reverse directed] ([c]L10) node[dot]{}-- ([c]L11) ;
\draw[thick, directed] ([c]L11)node[dot]{}-- ([c]L12) ;
\draw[thick] ([c]L12)node[dot]{}-- ([c]L10) ;
\draw[thick] ([c]L10)node[below]{4}-- ([c]L20) ;
\draw[thick,reverse directed] ([c]L12)node[left]{6}-- ([c]L20) ;
\draw[thick] ([c]L20)node[dot]{}-- ([c]L21) ;
\draw[thick,directed] ([c]L11)node[right]{5}-- ([c]L21) ;
\draw[thick] ([c]L10)-- ([c]L21) ;
\draw[thick,directed] ([c]L21)node[dot]{}-- ([c]L22) ;
\draw[thick,reverse directed] ([c]L22)node[dot]{}-- ([c]L20) ;
\draw[thick, directed] ([c]L12)-- ([c]L22) ;
\draw[thick,directed] ([c]L11)-- ([c]L22) ;
\draw[thick,reverse directed] ([c]L22)node[left]{7}-- ([c]L32)node[dot]{} ;
\draw[thick] ([c]L21)node[right]{2}-- ([c]L32)node[left]{3} ;
\draw[thick,directed] ([c]L20)node[left]{1}-- ([c]L32) ;
\end{scope}

\begin{scope}[every coordinate/.style={shift={(3,-2)}}]
\node[above] at ([c]a) {$E=$};
\draw[thick] ([c]L10) node[dot]{}-- ([c]L11) ;
\draw[thick] ([c]L11)node[dot]{}-- ([c]L12) ;
\draw[thick] ([c]L12)node[dot]{}-- ([c]L10) ;
\draw[thick] ([c]L10)node[below]{4}-- ([c]L20) ;
\draw[thick,directed] ([c]L12)node[left]{6}-- ([c]L20) ;
\draw[thick] ([c]L20)node[dot]{}-- ([c]L21) ;
\draw[thick] ([c]L11)node[right]{5}-- ([c]L21) ;
\draw[thick] ([c]L10)-- ([c]L21) ;
\draw[thick] ([c]L21)node[dot]{}-- ([c]L22) ;
\draw[thick] ([c]L22)node[dot]{}-- ([c]L20) ;
\draw[thick] ([c]L12)-- ([c]L22) ;
\draw[thick] ([c]L11)-- ([c]L22) ;
\draw[thick,directed] ([c]L22)node[left]{7}-- ([c]L32)node[dot]{} ;
\draw[thick] ([c]L21)node[right]{2}-- ([c]L32)node[left]{3} ;
\draw[thick,directed] ([c]L20)node[left]{1}-- ([c]L32) ;
\end{scope}

\begin{scope}[every coordinate/.style={shift={(6,-2)}}]
\node[above] at ([c]a) {$F=$};
\draw[thick] ([c]L10) node[dot]{}-- ([c]L11) ;
\draw[thick] ([c]L11)node[dot]{}-- ([c]L12) ;
\draw[thick] ([c]L12)node[dot]{}-- ([c]L10) ;
\draw[thick] ([c]L10)node[below]{4}-- ([c]L20) ;
\draw[thick,reverse directed] ([c]L12)node[left]{6}-- ([c]L20) ;
\draw[thick] ([c]L20)node[dot]{}-- ([c]L21) ;
\draw[thick] ([c]L11)node[right]{5}-- ([c]L21) ;
\draw[thick] ([c]L10)-- ([c]L21) ;
\draw[thick] ([c]L21)node[dot]{}-- ([c]L22) ;
\draw[thick,directed] ([c]L22)node[dot]{}-- ([c]L20) ;
\draw[thick,reverse directed] ([c]L12)-- ([c]L22) ;
\draw[thick] ([c]L11)-- ([c]L22) ;
\draw[thick,directed] ([c]L22)node[left]{7}-- ([c]L32)node[dot]{} ;
\draw[thick] ([c]L21)node[right]{2}-- ([c]L32)node[left]{3} ;
\draw[thick,directed] ([c]L20)node[left]{1}-- ([c]L32) ;
\end{scope}

\begin{scope}[every coordinate/.style={shift={(0,-4)}}]
\node[above] at ([c]a) {$G=$};
\draw[thick,reverse directed] ([c]L10) node[dot]{}-- ([c]L11) ;
\draw[thick] ([c]L11)node[dot]{}-- ([c]L12) ;
\draw[thick,reverse directed] ([c]L12)node[dot]{}-- ([c]L10) ;
\draw[thick,reverse directed] ([c]L10)node[below]{4}-- ([c]L20) ;
\draw[thick,reverse directed] ([c]L12)node[left]{6}-- ([c]L20) ;
\draw[thick] ([c]L20)node[dot]{}-- ([c]L21) ;
\draw[thick] ([c]L11)node[right]{5}-- ([c]L21) ;
\draw[thick] ([c]L10)-- ([c]L21) ;
\draw[thick] ([c]L21)node[dot]{}-- ([c]L22) ;
\draw[thick,reverse directed] ([c]L22)node[dot]{}-- ([c]L20) ;
\draw[thick,reverse directed] ([c]L12)-- ([c]L22) ;
\draw[thick] ([c]L11)-- ([c]L22) ;
\draw[thick,directed] ([c]L22)node[left]{7}-- ([c]L32)node[dot]{} ;
\draw[thick] ([c]L21)node[right]{2}-- ([c]L32)node[left]{3} ;
\draw[thick,directed] ([c]L20)node[left]{1}-- ([c]L32) ;
\end{scope}

\begin{scope}[every coordinate/.style={shift={(3,-4)}}]
\node[above] at ([c]a) {$H=$};
\draw[thick] ([c]L10) node[dot]{}-- ([c]L11) ;
\draw[thick,directed] ([c]L11)node[dot]{}-- ([c]L12) ;
\draw[thick] ([c]L12)node[dot]{}-- ([c]L10) ;
\draw[thick] ([c]L10)node[below]{4}-- ([c]L20) ;
\draw[thick,reverse directed] ([c]L12)node[left]{6}-- ([c]L20) ;
\draw[thick] ([c]L20)node[dot]{}-- ([c]L21) ;
\draw[thick,directed] ([c]L11)node[right]{5}-- ([c]L21) ;
\draw[thick] ([c]L10)-- ([c]L21) ;
\draw[thick] ([c]L21)node[dot]{}-- ([c]L22) ;
\draw[thick,directed] ([c]L22)node[dot]{}-- ([c]L20) ;
\draw[thick,reverse directed] ([c]L12)-- ([c]L22) ;
\draw[thick,reverse directed] ([c]L11)-- ([c]L22) ;
\draw[thick,directed] ([c]L22)node[left]{7}-- ([c]L32)node[dot]{} ;
\draw[thick] ([c]L21)node[right]{2}-- ([c]L32)node[left]{3} ;
\draw[thick,directed] ([c]L20)node[left]{1}-- ([c]L32) ;
\end{scope}

\begin{scope}[every coordinate/.style={shift={(6,-4)}}]
\node[above] at ([c]a) {$I=$};
\draw[thick] ([c]L10) node[dot]{}-- ([c]L11) ;
\draw[thick] ([c]L11)node[dot]{}-- ([c]L12) ;
\draw[thick] ([c]L12)node[dot]{}-- ([c]L10) ;
\draw[thick] ([c]L10)node[below]{4}-- ([c]L20) ;
\draw[thick,directed] ([c]L12)node[left]{6}-- ([c]L20) ;
\draw[thick,reverse directed] ([c]L20)node[dot]{}-- ([c]L21) ;
\draw[thick] ([c]L11)node[right]{5}-- ([c]L21) ;
\draw[thick] ([c]L10)-- ([c]L21) ;
\draw[thick] ([c]L21)node[dot]{}-- ([c]L22) ;
\draw[thick] ([c]L22)node[dot]{}-- ([c]L20) ;
\draw[thick,directed] ([c]L12)-- ([c]L22) ;
\draw[thick] ([c]L11)-- ([c]L22) ;
\draw[thick,directed] ([c]L22)node[left]{7}-- ([c]L32)node[dot]{} ;
\draw[thick] ([c]L21)node[right]{2}-- ([c]L32)node[left]{3} ;
\draw[thick,directed] ([c]L20)node[left]{1}-- ([c]L32) ;
\end{scope}

\begin{scope}[every coordinate/.style={shift={(0,-6)}}]
\node[above] at ([c]a) {$J=$};
\draw[thick] ([c]L10) node[dot]{}-- ([c]L11) ;
\draw[thick,reverse directed] ([c]L11)node[dot]{}-- ([c]L12) ;
\draw[thick,directed] ([c]L12)node[dot]{}-- ([c]L10) ;
\draw[thick] ([c]L10)node[below]{4}-- ([c]L20) ;
\draw[thick,directed] ([c]L12)node[left]{6}-- ([c]L20) ;
\draw[thick,directed] ([c]L20)node[dot]{}-- ([c]L21) ;
\draw[thick,directed] ([c]L11)node[right]{5}-- ([c]L21) ;
\draw[thick,directed] ([c]L10)-- ([c]L21) ;
\draw[thick,reverse directed] ([c]L21)node[dot]{}-- ([c]L22) ;
\draw[thick,directed] ([c]L22)node[dot]{}-- ([c]L20) ;
\draw[thick,directed] ([c]L12)-- ([c]L22) ;
\draw[thick] ([c]L11)-- ([c]L22) ;
\draw[thick,directed] ([c]L22)node[left]{7}-- ([c]L32)node[dot]{} ;
\draw[thick] ([c]L21)node[right]{2}-- ([c]L32)node[left]{3} ;
\draw[thick,directed] ([c]L20)node[left]{1}-- ([c]L32) ;
\end{scope}

\begin{scope}[every coordinate/.style={shift={(3,-6)}}]
\node[above] at ([c]a) {$K=$};
\draw[thick] ([c]L10) node[dot]{}-- ([c]L11) ;
\draw[thick] ([c]L11)node[dot]{}-- ([c]L12) ;
\draw[thick] ([c]L12)node[dot]{}-- ([c]L10) ;
\draw[thick,directed] ([c]L10)node[below]{4}-- ([c]L20) ;
\draw[thick,directed] ([c]L12)node[left]{6}-- ([c]L20) ;
\draw[thick,reverse directed] ([c]L20)node[dot]{}-- ([c]L21) ;
\draw[thick] ([c]L11)node[right]{5}-- ([c]L21) ;
\draw[thick] ([c]L10)-- ([c]L21) ;
\draw[thick] ([c]L21)node[dot]{}-- ([c]L22) ;
\draw[thick] ([c]L22)node[dot]{}-- ([c]L20) ;
\draw[thick,directed] ([c]L12)-- ([c]L22) ;
\draw[thick] ([c]L11)-- ([c]L22) ;
\draw[thick,directed] ([c]L22)node[left]{7}-- ([c]L32)node[dot]{} ;
\draw[thick,directed] ([c]L21)node[right]{2}-- ([c]L32)node[left]{3} ;
\draw[thick,directed] ([c]L20)node[left]{1}-- ([c]L32) ;
\end{scope}

\begin{scope}[every coordinate/.style={shift={(6,-6)}}]
\node[above] at ([c]a) {$L=$};
\draw[thick] ([c]L10) node[dot]{}-- ([c]L11) ;
\draw[thick] ([c]L11)node[dot]{}-- ([c]L12) ;
\draw[thick] ([c]L12)node[dot]{}-- ([c]L10) ;
\draw[thick,directed] ([c]L10)node[below]{4}-- ([c]L20) ;
\draw[thick,directed] ([c]L12)node[left]{6}-- ([c]L20) ;
\draw[thick,reverse directed] ([c]L20)node[dot]{}-- ([c]L21) ;
\draw[thick, directed] ([c]L11)node[right]{5}-- ([c]L21) ;
\draw[thick,directed] ([c]L10)-- ([c]L21) ;
\draw[thick] ([c]L21)node[dot]{}-- ([c]L22) ;
\draw[thick] ([c]L22)node[dot]{}-- ([c]L20) ;
\draw[thick,directed] ([c]L12)-- ([c]L22) ;
\draw[thick] ([c]L11)-- ([c]L22) ;
\draw[thick, directed] ([c]L22)node[left]{7}-- ([c]L32)node[dot]{} ;
\draw[thick,directed] ([c]L21)node[right]{2}-- ([c]L32)node[left]{3} ;
\draw[thick,directed] ([c]L20)node[left]{1}-- ([c]L32) ;
\end{scope}
\end{tikzpicture}
\caption{Partial orientations of the second graph in Figure~\ref{non-repr-induced-subgraphs}}\label{12-partially-orient-graphs}
\end{center}
\end{figure}

\subsection{An inductive argument proving Theorem~\ref{thm-3}}\label{inductive-arg-subsec}

To show that avoidance of the six graphs in Figure~\ref{non-repr-induced-subgraphs} as induced subgraphs is a sufficient condition for a GCCG with three sectors to be word-representable, we use the following approach.

Each triangulation $T$ of a GCCG having $i$ layers and no graph in Figure~\ref{non-repr-induced-subgraphs} as an induced subgraph is obtained from  a triangulation $T_0$ of a GCCG  having $i-1$ layers and no graph in Figure~\ref{non-repr-induced-subgraphs} as an induced subgraph by adding a new external layer $L_i$. Since the graphs in Figure~\ref{non-repr-induced-subgraphs} involve vertices from three layers (that is, four levels), to obtain all possible $T$, we need to control the two external layers (layers $L_{i-1}$ and $L_{i-2}$) in $T_0$. Further, if we assume existence of a semi-transitive orientation of $T_0$, which induces a semi-transitive orientation on  $L_{i-1}$ and $L_{i-2}$, we could try to extend such an orientation to a semi-transitive orientation of $T$ (making sure that no cycles or shortcuts emerge).

The next step is to compare the directed graphs induced by the layers $\{L_{i-1},L_{i}\}$ and $\{L_{i-2},L_{i-1}\}$. If these are the same directed graphs, then an inductive argument can be applied to extending the graph by new layers and proving that in each case a semi-transitive orientation exists giving word-representability by Theorem~\ref{semitra}. On the other hand, if  the graphs induced by the layers $\{L_{i-1},L_{i}\}$ and $\{L_{i-2},L_{i-1}\}$ are different, we need to replace the oriented layers $\{L_{i-2},L_{i-1}\}$ by  $\{L_{i-1},L_{i}\}$ and repeat the procedure described above again. Namely, we need to extend the graph by another layer $L_{i+1}$, then try to extend the existing semi-transitive orientation to this layer, and compare two external layers $\{L_{i},L_{i+1}\}$ with already considered orientations of two external layers with a hope to meet the same directed graph.

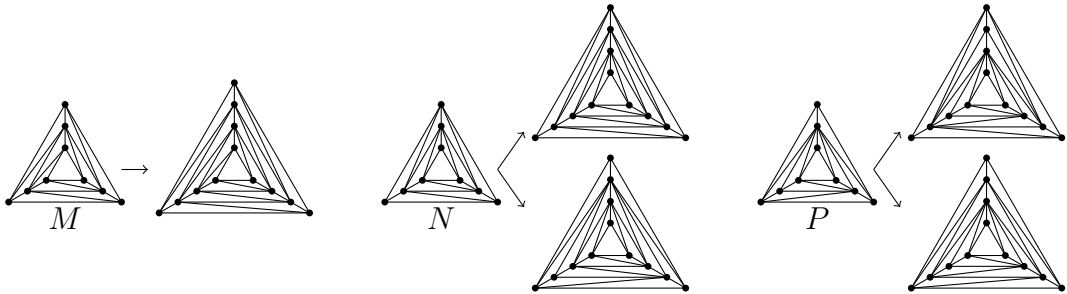
\begin{figure}[!htbp]
\begin{center}
\begin{tikzpicture}[scale=0.5,dot/.style={circle,inner sep=2pt,fill,name=#1}]
\def\dist{3.5}
\def\xa{4.5}
\def\xb{5.5}
\def\y{2}
\coordinate (Arrow) at (1.5, 0);
\coordinate (lab) at (0,-1.3) ; 
\coordinate (L10) at (-0.5, -0.289);
\coordinate (L11) at (0.5, -0.289);
\coordinate (L12) at (0, 0.577);
\coordinate (L20) at (-1, -0.577);
\coordinate (L21) at (1, -0.577);
\coordinate (L22) at (0, 1.154);
\coordinate (L30) at (-1.5, -0.867);
\coordinate (L31) at (1.5, -0.867);
\coordinate (L32) at (0, 1.731);
\coordinate (L40) at (-2., -1.156);
\coordinate (L41) at (2., -1.156);
\coordinate (L42) at (0, 2.308);


\begin{scope}[every coordinate/.style={shift={(0,0)}}]
\node at ([c]lab) {$M$};
\fill[black!100] ([c]L10) circle(0.5ex)
		 ([c]L11) circle(0.5ex)
		 ([c]L12) circle(0.5ex)
		 ([c]L20) circle(0.5ex)
		 ([c]L21) circle(0.5ex)
		 ([c]L22) circle(0.5ex)
		 ([c]L30) circle(0.5ex)
		 ([c]L31) circle(0.5ex)
		 ([c]L32) circle(0.5ex)
		 ;
\draw[->] ([c]Arrow)--+(0.7,0);
\draw[] ([c]L10)-- ([c]L11) ;
\draw[] ([c]L11)-- ([c]L12) ;
\draw[] ([c]L12)-- ([c]L10) ;
\draw[] ([c]L20)-- ([c]L21) ;
\draw[] ([c]L21)-- ([c]L22) ;
\draw[] ([c]L22)-- ([c]L20) ;
\draw[] ([c]L30)-- ([c]L31) ;
\draw[] ([c]L31)-- ([c]L32) ;
\draw[] ([c]L32)-- ([c]L30) ;
\draw[] ([c]L10)-- ([c]L20) ;
\draw[] ([c]L11)-- ([c]L21) ;
\draw[] ([c]L12)-- ([c]L22) ;
\draw[] ([c]L20)-- ([c]L30) ;
\draw[] ([c]L21)-- ([c]L31) ;
\draw[] ([c]L22)-- ([c]L32) ;
\draw[] ([c]L10)-- ([c]L21) ;
\draw[] ([c]L11)-- ([c]L22) ;
\draw[] ([c]L12)-- ([c]L20) ;
\draw[] ([c]L20)-- ([c]L31) ;
\draw[] ([c]L21)-- ([c]L32) ;
\draw[] ([c]L22)-- ([c]L30) ;
\end{scope}

\begin{scope}[every coordinate/.style={shift={(\xa,0)}}]
\fill[black!100] ([c]L10) circle(0.5ex)
		 ([c]L11) circle(0.5ex)
		 ([c]L12) circle(0.5ex)
		 ([c]L20) circle(0.5ex)
		 ([c]L21) circle(0.5ex)
		 ([c]L22) circle(0.5ex)
		 ([c]L30) circle(0.5ex)
		 ([c]L31) circle(0.5ex)
		 ([c]L32) circle(0.5ex)
		 ([c]L40) circle(0.5ex)
		 ([c]L41) circle(0.5ex)
		 ([c]L42) circle(0.5ex)
		 ;

\draw[] ([c]L10)-- ([c]L11) ;
\draw[] ([c]L11)-- ([c]L12) ;
\draw[] ([c]L12)-- ([c]L10) ;
\draw[] ([c]L20)-- ([c]L21) ;
\draw[] ([c]L21)-- ([c]L22) ;
\draw[] ([c]L22)-- ([c]L20) ;
\draw[] ([c]L30)-- ([c]L31) ;
\draw[] ([c]L31)-- ([c]L32) ;
\draw[] ([c]L32)-- ([c]L30) ;
\draw[] ([c]L10)-- ([c]L20) ;
\draw[] ([c]L11)-- ([c]L21) ;
\draw[] ([c]L12)-- ([c]L22) ;
\draw[] ([c]L20)-- ([c]L30) ;
\draw[] ([c]L21)-- ([c]L31) ;
\draw[] ([c]L22)-- ([c]L32) ;
\draw[] ([c]L10)-- ([c]L21) ;
\draw[] ([c]L11)-- ([c]L22) ;
\draw[] ([c]L12)-- ([c]L20) ;
\draw[] ([c]L20)-- ([c]L31) ;
\draw[] ([c]L21)-- ([c]L32) ;
\draw[] ([c]L22)-- ([c]L30) ;
\draw[] ([c]L40)-- ([c]L41) ;
\draw[] ([c]L41)-- ([c]L42) ;
\draw[] ([c]L42)-- ([c]L40) ;
\draw[] ([c]L30)-- ([c]L40) ;
\draw[] ([c]L31)-- ([c]L41) ;
\draw[] ([c]L32)-- ([c]L42) ;
\draw[] ([c]L30)-- ([c]L41) ;
\draw[] ([c]L31)-- ([c]L42) ;
\draw[] ([c]L32)-- ([c]L40) ;
\end{scope}
\begin{scope}[every coordinate/.style={shift={(\xa+\xb,0)}}]
\node at ([c]lab) {$N$};
\fill[black!100] ([c]L10) circle(0.5ex)
		 ([c]L11) circle(0.5ex)
		 ([c]L12) circle(0.5ex)
		 ([c]L20) circle(0.5ex)
		 ([c]L21) circle(0.5ex)
		 ([c]L22) circle(0.5ex)
		 ([c]L30) circle(0.5ex)
		 ([c]L31) circle(0.5ex)
		 ([c]L32) circle(0.5ex)
		 ;
\draw[->] ([c]Arrow)--+(0.7,1);
\draw[->] ([c]Arrow)--+(0.7,-1);
\draw[] ([c]L10)-- ([c]L11) ;
\draw[] ([c]L11)-- ([c]L12) ;
\draw[] ([c]L12)-- ([c]L10) ;
\draw[] ([c]L20)-- ([c]L21) ;
\draw[] ([c]L21)-- ([c]L22) ;
\draw[] ([c]L22)-- ([c]L20) ;
\draw[] ([c]L30)-- ([c]L31) ;
\draw[] ([c]L31)-- ([c]L32) ;
\draw[] ([c]L32)-- ([c]L30) ;
\draw[] ([c]L10)-- ([c]L20) ;
\draw[] ([c]L11)-- ([c]L21) ;
\draw[] ([c]L12)-- ([c]L22) ;
\draw[] ([c]L20)-- ([c]L30) ;
\draw[] ([c]L21)-- ([c]L31) ;
\draw[] ([c]L22)-- ([c]L32) ;
\draw[] ([c]L10)-- ([c]L21) ;
\draw[] ([c]L11)-- ([c]L22) ;
\draw[] ([c]L10)-- ([c]L22) ;
\draw[] ([c]L20)-- ([c]L31) ;
\draw[] ([c]L21)-- ([c]L32) ;
\draw[] ([c]L20)-- ([c]L32) ;	
\end{scope}

\begin{scope}[every coordinate/.style={shift={(2*\xa+\xb,\y)}}]
\fill[black!100] ([c]L10) circle(0.5ex)
		 ([c]L11) circle(0.5ex)
		 ([c]L12) circle(0.5ex)
		 ([c]L20) circle(0.5ex)
		 ([c]L21) circle(0.5ex)
		 ([c]L22) circle(0.5ex)
		 ([c]L30) circle(0.5ex)
		 ([c]L31) circle(0.5ex)
		 ([c]L32) circle(0.5ex)
		 ([c]L40) circle(0.5ex)
		 ([c]L41) circle(0.5ex)
		 ([c]L42) circle(0.5ex)
		 ;
\draw[] ([c]L10)-- ([c]L11) ;
\draw[] ([c]L11)-- ([c]L12) ;
\draw[] ([c]L12)-- ([c]L10) ;
\draw[] ([c]L20)-- ([c]L21) ;
\draw[] ([c]L21)-- ([c]L22) ;
\draw[] ([c]L22)-- ([c]L20) ;
\draw[] ([c]L30)-- ([c]L31) ;
\draw[] ([c]L31)-- ([c]L32) ;
\draw[] ([c]L32)-- ([c]L30) ;
\draw[] ([c]L10)-- ([c]L20) ;
\draw[] ([c]L11)-- ([c]L21) ;
\draw[] ([c]L12)-- ([c]L22) ;
\draw[] ([c]L20)-- ([c]L30) ;
\draw[] ([c]L21)-- ([c]L31) ;
\draw[] ([c]L22)-- ([c]L32) ;
\draw[] ([c]L10)-- ([c]L21) ;
\draw[] ([c]L11)-- ([c]L22) ;
\draw[] ([c]L10)-- ([c]L22) ;
\draw[] ([c]L20)-- ([c]L31) ;
\draw[] ([c]L21)-- ([c]L32) ;
\draw[] ([c]L20)-- ([c]L32) ;
\draw[] ([c]L40)-- ([c]L41) ;
\draw[] ([c]L41)-- ([c]L42) ;
\draw[] ([c]L42)-- ([c]L40) ;
\draw[] ([c]L30)-- ([c]L40) ;
\draw[] ([c]L31)-- ([c]L41) ;
\draw[] ([c]L32)-- ([c]L42) ;
\draw[] ([c]L30)-- ([c]L41) ;
\draw[] ([c]L31)-- ([c]L42) ;
\draw[] ([c]L30)-- ([c]L42) ;
		
\end{scope}
\begin{scope}[every coordinate/.style={shift={(2*\xa+\xb,-1*\y)}}]
\fill[black!100] ([c]L10) circle(0.5ex)
		 ([c]L11) circle(0.5ex)
		 ([c]L12) circle(0.5ex)
		 ([c]L20) circle(0.5ex)
		 ([c]L21) circle(0.5ex)
		 ([c]L22) circle(0.5ex)
		 ([c]L30) circle(0.5ex)
		 ([c]L31) circle(0.5ex)
		 ([c]L32) circle(0.5ex)
		 ([c]L40) circle(0.5ex)
		 ([c]L41) circle(0.5ex)
		 ([c]L42) circle(0.5ex)
		 ;

\draw[] ([c]L10)-- ([c]L11) ;
\draw[] ([c]L11)-- ([c]L12) ;
\draw[] ([c]L12)-- ([c]L10) ;
\draw[] ([c]L20)-- ([c]L21) ;
\draw[] ([c]L21)-- ([c]L22) ;
\draw[] ([c]L22)-- ([c]L20) ;
\draw[] ([c]L30)-- ([c]L31) ;
\draw[] ([c]L31)-- ([c]L32) ;
\draw[] ([c]L32)-- ([c]L30) ;
\draw[] ([c]L10)-- ([c]L20) ;
\draw[] ([c]L11)-- ([c]L21) ;
\draw[] ([c]L12)-- ([c]L22) ;
\draw[] ([c]L20)-- ([c]L30) ;
\draw[] ([c]L21)-- ([c]L31) ;
\draw[] ([c]L22)-- ([c]L32) ;
\draw[] ([c]L10)-- ([c]L21) ;
\draw[] ([c]L11)-- ([c]L22) ;
\draw[] ([c]L10)-- ([c]L22) ;
\draw[] ([c]L20)-- ([c]L31) ;
\draw[] ([c]L21)-- ([c]L32) ;
\draw[] ([c]L20)-- ([c]L32) ;
\draw[] ([c]L40)-- ([c]L41) ;
\draw[] ([c]L41)-- ([c]L42) ;
\draw[] ([c]L42)-- ([c]L40) ;
\draw[] ([c]L30)-- ([c]L40) ;
\draw[] ([c]L31)-- ([c]L41) ;
\draw[] ([c]L32)-- ([c]L42) ;
\draw[] ([c]L31)-- ([c]L40) ;
\draw[] ([c]L32)-- ([c]L41) ;
\draw[] ([c]L32)-- ([c]L40) ;		
\end{scope}

\begin{scope}[every coordinate/.style={shift={(2*\xa+2*\xb,0)}}]
\node at ([c]lab) {$P$};
\fill[black!100] ([c]L10) circle(0.5ex)
		 ([c]L11) circle(0.5ex)
		 ([c]L12) circle(0.5ex)
		 ([c]L20) circle(0.5ex)
		 ([c]L21) circle(0.5ex)
		 ([c]L22) circle(0.5ex)
		 ([c]L30) circle(0.5ex)
		 ([c]L31) circle(0.5ex)
		 ([c]L32) circle(0.5ex)
		 ;
\draw[->] ([c]Arrow)--+(0.7,1);
\draw[->] ([c]Arrow)--+(0.7,-1);
\draw[] ([c]L10)-- ([c]L11) ;
\draw[] ([c]L11)-- ([c]L12) ;
\draw[] ([c]L12)-- ([c]L10) ;
\draw[] ([c]L20)-- ([c]L21) ;
\draw[] ([c]L21)-- ([c]L22) ;
\draw[] ([c]L22)-- ([c]L20) ;
\draw[] ([c]L30)-- ([c]L31) ;
\draw[] ([c]L31)-- ([c]L32) ;
\draw[] ([c]L32)-- ([c]L30) ;
\draw[] ([c]L10)-- ([c]L20) ;
\draw[] ([c]L11)-- ([c]L21) ;
\draw[] ([c]L12)-- ([c]L22) ;
\draw[] ([c]L20)-- ([c]L30) ;
\draw[] ([c]L21)-- ([c]L31) ;
\draw[] ([c]L22)-- ([c]L32) ;
\draw[] ([c]L10)-- ([c]L21) ;
\draw[] ([c]L11)-- ([c]L22) ;
\draw[] ([c]L10)-- ([c]L22) ;
\draw[] ([c]L21)-- ([c]L30) ;
\draw[] ([c]L22)-- ([c]L31) ;
\draw[] ([c]L22)-- ([c]L30) ;		
\end{scope}
\begin{scope}[every coordinate/.style={shift={(3*\xa+2*\xb,\y)}}]
\fill[black!100] ([c]L10) circle(0.5ex)
		 ([c]L11) circle(0.5ex)
		 ([c]L12) circle(0.5ex)
		 ([c]L20) circle(0.5ex)
		 ([c]L21) circle(0.5ex)
		 ([c]L22) circle(0.5ex)
		 ([c]L30) circle(0.5ex)
		 ([c]L31) circle(0.5ex)
		 ([c]L32) circle(0.5ex)
		 ([c]L40) circle(0.5ex)
		 ([c]L41) circle(0.5ex)
		 ([c]L42) circle(0.5ex)
		 ;
\draw[] ([c]L10)-- ([c]L11) ;
\draw[] ([c]L11)-- ([c]L12) ;
\draw[] ([c]L12)-- ([c]L10) ;
\draw[] ([c]L20)-- ([c]L21) ;
\draw[] ([c]L21)-- ([c]L22) ;
\draw[] ([c]L22)-- ([c]L20) ;
\draw[] ([c]L30)-- ([c]L31) ;
\draw[] ([c]L31)-- ([c]L32) ;
\draw[] ([c]L32)-- ([c]L30) ;
\draw[] ([c]L10)-- ([c]L20) ;
\draw[] ([c]L11)-- ([c]L21) ;
\draw[] ([c]L12)-- ([c]L22) ;
\draw[] ([c]L20)-- ([c]L30) ;
\draw[] ([c]L21)-- ([c]L31) ;
\draw[] ([c]L22)-- ([c]L32) ;
\draw[] ([c]L10)-- ([c]L21) ;
\draw[] ([c]L11)-- ([c]L22) ;
\draw[] ([c]L10)-- ([c]L22) ;
\draw[] ([c]L21)-- ([c]L30) ;
\draw[] ([c]L22)-- ([c]L31) ;
\draw[] ([c]L22)-- ([c]L30) ;
\draw[] ([c]L40)-- ([c]L41) ;
\draw[] ([c]L41)-- ([c]L42) ;
\draw[] ([c]L42)-- ([c]L40) ;
\draw[] ([c]L30)-- ([c]L40) ;
\draw[] ([c]L31)-- ([c]L41) ;
\draw[] ([c]L32)-- ([c]L42) ;
\draw[] ([c]L30)-- ([c]L41) ;
\draw[] ([c]L31)-- ([c]L42) ;
\draw[] ([c]L30)-- ([c]L42) ;		
\end{scope}
\begin{scope}[every coordinate/.style={shift={(3*\xa+2*\xb,-1*\y)}}]
\fill[black!100] ([c]L10) circle(0.5ex)
		 ([c]L11) circle(0.5ex)
		 ([c]L12) circle(0.5ex)
		 ([c]L20) circle(0.5ex)
		 ([c]L21) circle(0.5ex)
		 ([c]L22) circle(0.5ex)
		 ([c]L30) circle(0.5ex)
		 ([c]L31) circle(0.5ex)
		 ([c]L32) circle(0.5ex)
		 ([c]L40) circle(0.5ex)
		 ([c]L41) circle(0.5ex)
		 ([c]L42) circle(0.5ex)
		 ;
\draw[] ([c]L10)-- ([c]L11) ;
\draw[] ([c]L11)-- ([c]L12) ;
\draw[] ([c]L12)-- ([c]L10) ;
\draw[] ([c]L20)-- ([c]L21) ;
\draw[] ([c]L21)-- ([c]L22) ;
\draw[] ([c]L22)-- ([c]L20) ;
\draw[] ([c]L30)-- ([c]L31) ;
\draw[] ([c]L31)-- ([c]L32) ;
\draw[] ([c]L32)-- ([c]L30) ;
\draw[] ([c]L10)-- ([c]L20) ;
\draw[] ([c]L11)-- ([c]L21) ;
\draw[] ([c]L12)-- ([c]L22) ;
\draw[] ([c]L20)-- ([c]L30) ;
\draw[] ([c]L21)-- ([c]L31) ;
\draw[] ([c]L22)-- ([c]L32) ;
\draw[] ([c]L10)-- ([c]L21) ;
\draw[] ([c]L11)-- ([c]L22) ;
\draw[] ([c]L10)-- ([c]L22) ;
\draw[] ([c]L21)-- ([c]L30) ;
\draw[] ([c]L22)-- ([c]L31) ;
\draw[] ([c]L22)-- ([c]L30) ;
\draw[] ([c]L40)-- ([c]L41) ;
\draw[] ([c]L41)-- ([c]L42) ;
\draw[] ([c]L42)-- ([c]L40) ;
\draw[] ([c]L30)-- ([c]L40) ;
\draw[] ([c]L31)-- ([c]L41) ;
\draw[] ([c]L32)-- ([c]L42) ;
\draw[] ([c]L31)-- ([c]L40) ;
\draw[] ([c]L32)-- ([c]L41) ;
\draw[] ([c]L32)-- ([c]L40) ;		
\end{scope}
\end{tikzpicture}
\caption{The extensions of the graphs $M$, $N$ and $P$}\label{3-possible-cases-andtheir-5-exten}
\end{center}
\end{figure}

The base for our inductive proof is the three graphs, $M$, $N$ and $P$ in Figure~\ref{3-possible-cases-andtheir-5-exten}, which are the only non-isomorphic  triangulations of the GCCG on nine vertices containing no graph in Figure~\ref{non-repr-induced-subgraphs} as an induced subgraph. Each of these graphs can be semi-transitively oriented as shown in Figure~\ref{induction-base-case}. Note that we provide two semi-transitive orientations for the graphs $N$ and $P$, which is essentially in our inductive argument. Also, note that each triangulation of a GCCG with three sectors on less than nine vertices can be oriented semi-transitively (just remove the external layer in the graphs in Figure~\ref{induction-base-case}).

Further, we note that there are only five ways in total in which the graphs $M$, $N$ and $P$ can be extended by one more (external) layer if the graphs in Figure~\ref{non-repr-induced-subgraphs} are to be avoided as an induced subgraph. The extensions are shown in~\ref{3-possible-cases-andtheir-5-exten}.

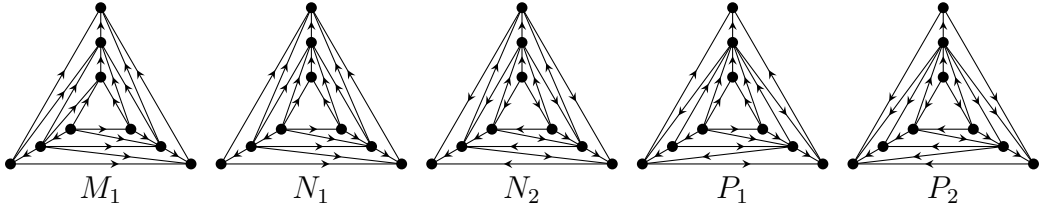
\begin{figure}[!htbp]
\begin{center}
\begin{tikzpicture}[scale=0.8,dot/.style={circle,inner sep=1.5pt,fill,name=#1}]
\def\dist{3.5}
\coordinate (lab) at (0,-1.3) ;
\coordinate (L1) at (-0.5, -0.289) ;
\coordinate (L2) at (0.5, -0.289) ;
\coordinate (L3) at (0, 0.577) ;
\coordinate (L4) at (-1, -0.577) ;
\coordinate (L5) at (1, -0.577) ;
\coordinate (L6) at (0, 1.154) ;
\coordinate (L7) at (-1.5, -0.867) ;
\coordinate (L8) at (1.5, -0.867) ;
\coordinate (L9) at (0, 1.731) ;


\begin{scope}[every coordinate/.style={shift={(0,0)}}]
\node at ([c]lab) {$M_1$};
\fill[black!100] ([c]L1) circle(0.5ex)
		 ([c]L2) circle(0.5ex)
		 ([c]L3) circle(0.5ex)
		 ([c]L4) circle(0.5ex)
		 ([c]L5) circle(0.5ex)
		 ([c]L6) circle(0.5ex)
		 ([c]L7) circle(0.5ex)
		 ([c]L8) circle(0.5ex)
		 ([c]L9) circle(0.5ex)
		 ;
\draw[,directed] ([c]L1) -- ([c]L2) ;
\draw[,directed] ([c]L1) -- ([c]L3) ;
\draw[,directed] ([c]L1) -- ([c]L4) ;
\draw[,directed] ([c]L1) -- ([c]L5) ;
\draw[,directed] ([c]L2) -- ([c]L3) ;
\draw[,directed] ([c]L2) -- ([c]L5) ;
\draw[,directed] ([c]L2) -- ([c]L6) ;
\draw[,directed] ([c]L3) -- ([c]L6) ;
\draw[,directed] ([c]L4) -- ([c]L3) ;
\draw[,directed] ([c]L4) -- ([c]L5) ;
\draw[,directed] ([c]L4) -- ([c]L6) ;
\draw[,directed] ([c]L4) -- ([c]L7) ;
\draw[,directed] ([c]L4) -- ([c]L8) ;
\draw[,directed] ([c]L5) -- ([c]L6) ;
\draw[,directed] ([c]L5) -- ([c]L8) ;
\draw[,directed] ([c]L5) -- ([c]L9) ;
\draw[,directed] ([c]L6) -- ([c]L9) ;
\draw[,directed] ([c]L7) -- ([c]L6) ;
\draw[,directed] ([c]L7) -- ([c]L8) ;
\draw[,directed] ([c]L7) -- ([c]L9) ;
\draw[,directed] ([c]L8) -- ([c]L9) ;
\end{scope}

\begin{scope}[every coordinate/.style={shift={(2*\dist,0)}}]
\node at ([c]lab) {$N_2$};
\fill[black!100] ([c]L1) circle(0.5ex)
		 ([c]L2) circle(0.5ex)
		 ([c]L3) circle(0.5ex)
		 ([c]L4) circle(0.5ex)
		 ([c]L5) circle(0.5ex)
		 ([c]L6) circle(0.5ex)
		 ([c]L7) circle(0.5ex)
		 ([c]L8) circle(0.5ex)
		 ([c]L9) circle(0.5ex)
		 ;
\draw[,directed] ([c]L1) -- ([c]L4) ;
\draw[,directed] ([c]L1) -- ([c]L5) ;
\draw[,directed] ([c]L1) -- ([c]L6) ;
\draw[,directed] ([c]L2) -- ([c]L1) ;
\draw[,directed] ([c]L2) -- ([c]L5) ;
\draw[,directed] ([c]L2) -- ([c]L6) ;
\draw[,directed] ([c]L3) -- ([c]L1) ;
\draw[,directed] ([c]L3) -- ([c]L2) ;
\draw[,directed] ([c]L3) -- ([c]L6) ;
\draw[,directed] ([c]L4) -- ([c]L7) ;
\draw[,directed] ([c]L4) -- ([c]L8) ;
\draw[,directed] ([c]L4) -- ([c]L9) ;
\draw[,directed] ([c]L5) -- ([c]L4) ;
\draw[,directed] ([c]L5) -- ([c]L8) ;
\draw[,directed] ([c]L5) -- ([c]L9) ;
\draw[,directed] ([c]L6) -- ([c]L4) ;
\draw[,directed] ([c]L6) -- ([c]L5) ;
\draw[,directed] ([c]L6) -- ([c]L9) ;
\draw[,directed] ([c]L8) -- ([c]L7) ;
\draw[,directed] ([c]L9) -- ([c]L7) ;
\draw[,directed] ([c]L9) -- ([c]L8) ;
\end{scope}

\begin{scope}[every coordinate/.style={shift={(\dist,0)}}]
\node at ([c]lab) {$N_1$};
\fill[black!100] ([c]L1) circle(0.5ex)
		 ([c]L2) circle(0.5ex)
		 ([c]L3) circle(0.5ex)
		 ([c]L4) circle(0.5ex)
		 ([c]L5) circle(0.5ex)
		 ([c]L6) circle(0.5ex)
		 ([c]L7) circle(0.5ex)
		 ([c]L8) circle(0.5ex)
		 ([c]L9) circle(0.5ex)
		 ;
\draw[,directed] ([c]L1) -- ([c]L2) ;
\draw[,directed] ([c]L1) -- ([c]L3) ;
\draw[,directed] ([c]L1) -- ([c]L4) ;
\draw[,directed] ([c]L1) -- ([c]L5) ;
\draw[,directed] ([c]L1) -- ([c]L6) ;
\draw[,directed] ([c]L2) -- ([c]L3) ;
\draw[,directed] ([c]L2) -- ([c]L5) ;
\draw[,directed] ([c]L2) -- ([c]L6) ;
\draw[,directed] ([c]L3) -- ([c]L6) ;
\draw[,directed] ([c]L4) -- ([c]L5) ;
\draw[,directed] ([c]L4) -- ([c]L6) ;
\draw[,directed] ([c]L4) -- ([c]L7) ;
\draw[,directed] ([c]L4) -- ([c]L8) ;
\draw[,directed] ([c]L4) -- ([c]L9) ;
\draw[,directed] ([c]L5) -- ([c]L6) ;
\draw[,directed] ([c]L5) -- ([c]L8) ;
\draw[,directed] ([c]L5) -- ([c]L9) ;
\draw[,directed] ([c]L6) -- ([c]L9) ;
\draw[,directed] ([c]L7) -- ([c]L8) ;
\draw[,directed] ([c]L7) -- ([c]L9) ;
\draw[,directed] ([c]L8) -- ([c]L9) ;
\end{scope}

\begin{scope}[every coordinate/.style={shift={(3*\dist,0)}}]
\node at ([c]lab) {$P_1$};
\fill[black!100] ([c]L1) circle(0.5ex)
		 ([c]L2) circle(0.5ex)
		 ([c]L3) circle(0.5ex)
		 ([c]L4) circle(0.5ex)
		 ([c]L5) circle(0.5ex)
		 ([c]L6) circle(0.5ex)
		 ([c]L7) circle(0.5ex)
		 ([c]L8) circle(0.5ex)
		 ([c]L9) circle(0.5ex)
		 ;
\draw[,directed] ([c]L1) -- ([c]L2) ;
\draw[,directed] ([c]L1) -- ([c]L3) ;
\draw[,directed] ([c]L1) -- ([c]L4) ;
\draw[,directed] ([c]L1) -- ([c]L5) ;
\draw[,directed] ([c]L1) -- ([c]L6) ;
\draw[,directed] ([c]L2) -- ([c]L3) ;
\draw[,directed] ([c]L2) -- ([c]L5) ;
\draw[,directed] ([c]L2) -- ([c]L6) ;
\draw[,directed] ([c]L3) -- ([c]L6) ;
\draw[,directed] ([c]L4) -- ([c]L5) ;
\draw[,directed] ([c]L4) -- ([c]L6) ;
\draw[,directed] ([c]L4) -- ([c]L7) ;
\draw[,directed] ([c]L5) -- ([c]L6) ;
\draw[,directed] ([c]L5) -- ([c]L7) ;
\draw[,directed] ([c]L5) -- ([c]L8) ;
\draw[,directed] ([c]L6) -- ([c]L7) ;
\draw[,directed] ([c]L6) -- ([c]L8) ;
\draw[,directed] ([c]L6) -- ([c]L9) ;
\draw[,directed] ([c]L7) -- ([c]L8) ;
\draw[,directed] ([c]L7) -- ([c]L9) ;
\draw[,directed] ([c]L8) -- ([c]L9) ;
\end{scope}

\begin{scope}[every coordinate/.style={shift={(4*\dist,0)}}]
\node at ([c]lab) {$P_2$};
\fill[black!100] ([c]L1) circle(0.5ex)
		 ([c]L2) circle(0.5ex)
		 ([c]L3) circle(0.5ex)
		 ([c]L4) circle(0.5ex)
		 ([c]L5) circle(0.5ex)
		 ([c]L6) circle(0.5ex)
		 ([c]L7) circle(0.5ex)
		 ([c]L8) circle(0.5ex)
		 ([c]L9) circle(0.5ex)
		 ;
\draw[,directed] ([c]L1) -- ([c]L4) ;
\draw[,directed] ([c]L1) -- ([c]L5) ;
\draw[,directed] ([c]L1) -- ([c]L6) ;
\draw[,directed] ([c]L2) -- ([c]L1) ;
\draw[,directed] ([c]L2) -- ([c]L5) ;
\draw[,directed] ([c]L2) -- ([c]L6) ;
\draw[,directed] ([c]L3) -- ([c]L1) ;
\draw[,directed] ([c]L3) -- ([c]L2) ;
\draw[,directed] ([c]L3) -- ([c]L6) ;
\draw[,directed] ([c]L4) -- ([c]L7) ;
\draw[,directed] ([c]L5) -- ([c]L4) ;
\draw[,directed] ([c]L5) -- ([c]L7) ;
\draw[,directed] ([c]L5) -- ([c]L8) ;
\draw[,directed] ([c]L6) -- ([c]L4) ;
\draw[,directed] ([c]L6) -- ([c]L5) ;
\draw[,directed] ([c]L6) -- ([c]L7) ;
\draw[,directed] ([c]L6) -- ([c]L8) ;
\draw[,directed] ([c]L6) -- ([c]L9) ;
\draw[,directed] ([c]L8) -- ([c]L7) ;
\draw[,directed] ([c]L9) -- ([c]L7) ;
\draw[,directed] ([c]L9) -- ([c]L8) ;
\end{scope}
\end{tikzpicture}
\caption{Semi-transitive orientations of the graphs $M$, $N$ and $P$}\label{induction-base-case}
\end{center}
\end{figure}

We now make an inductive hypothesis that each triangulation of a GCCG with three sectors and $n$ layers having no graph in Figure~\ref{non-repr-induced-subgraphs} as an induced subgraph can be oriented semi-transitive so that two external layers form the same directed graph as one of the directed graphs in Figure~\ref{induction-base-case}.

We will next prove the statement for $n+1$ layers by adding to all of the graphs in Figure~\ref{induction-base-case}  one more layer, in all possible ways, and extending the orientation of the resulting partially oriented graph to a semi-transitive orientation.  It is important to note that in each case below, it will follow from our way to extend orientations that no cycle or shortcut will be possible involving vertices on newly added level $C_{n+1}$ {\em and} vertices on levels $C_{n-3}$, $C_{n-4},\ldots, C_0$.

$M$ has a unique extension by an extra layer, and the orientation of $M_1$ can be extended to that shown in Figure~\ref{M1-extension}. Note that the two external layers of the graph in Figure~\ref{M1-extension} form $M_1$, as desired.

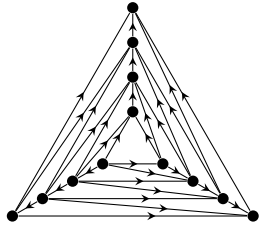
\begin{figure}[!htbp]
\begin{center}
\begin{tikzpicture}[scale=0.8,dot/.style={circle,inner sep=1.5pt,fill,name=#1}]
\node[dot=L1] at (-0.5, -0.289) {} ;
\node[dot=L2] at (0.5, -0.289) {} ;
\node[dot=L3] at (0, 0.577) {} ;
\node[dot=L4] at (-1, -0.577) {} ;
\node[dot=L5] at (1, -0.577) {} ;
\node[dot=L6] at (0, 1.154) {} ;
\node[dot=L7] at (-1.5, -0.867) {} ;
\node[dot=L8] at (1.5, -0.867) {} ;
\node[dot=L9] at (0, 1.731) {} ;
\node[dot=L10] at (-2., -1.156) {} ;
\node[dot=L11] at (2., -1.156) {} ;
\node[dot=L12] at (0, 2.308) {} ;
\draw[,directed] (L1) -- (L2) ;
\draw[,directed] (L1) -- (L3) ;
\draw[,directed] (L1) -- (L4) ;
\draw[,directed] (L1) -- (L5) ;
\draw[,directed] (L2) -- (L3) ;
\draw[,directed] (L2) -- (L5) ;
\draw[,directed] (L2) -- (L6) ;
\draw[,directed] (L3) -- (L6) ;
\draw[,directed] (L4) -- (L3) ;
\draw[,directed] (L4) -- (L5) ;
\draw[,directed] (L4) -- (L6) ;
\draw[,directed] (L4) -- (L7) ;
\draw[,directed] (L4) -- (L8) ;
\draw[,directed] (L5) -- (L6) ;
\draw[,directed] (L5) -- (L8) ;
\draw[,directed] (L5) -- (L9) ;
\draw[,directed] (L6) -- (L9) ;
\draw[,directed] (L7) -- (L6) ;
\draw[,directed] (L7) -- (L8) ;
\draw[,directed] (L7) -- (L9) ;
\draw[,directed] (L7) -- (L10) ;
\draw[,directed] (L7) -- (L11) ;
\draw[,directed] (L8) -- (L9) ;
\draw[,directed] (L8) -- (L11) ;
\draw[,directed] (L8) -- (L12) ;
\draw[,directed] (L9) -- (L12) ;
\draw[,directed] (L10) -- (L9) ;
\draw[,directed] (L10) -- (L11) ;
\draw[,directed] (L10) -- (L12) ;
\draw[,directed] (L11) -- (L12) ;
\end{tikzpicture}

\caption{The extension of $M_1$}\label{M1-extension}
\end{center}
\end{figure}

$N_1$ has two possible extensions by an extra layer, but both of them can be extended to semi-transitive orientation: see $N_{11}$ and $N_{12}$ in Figure~\ref{N1-N2-extensions}. Two external layers of $N_{11}$ and $N_{12}$ form $N_1$ and $P_1$, respectively, as desired.

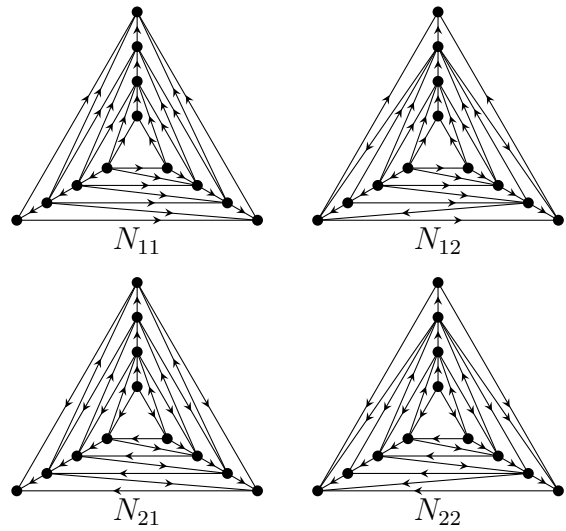
\begin{figure}[!htbp]
\begin{center}
\begin{tikzpicture}[scale=0.8,dot/.style={circle,inner sep=1pt,fill,name=#1}]
\def\dist{3.5}
\def\x{5}
\def\y{4.5}
\coordinate (lab) at (0,-1.5) ; 
\coordinate (L1) at (-0.5, -0.289);
\coordinate (L2) at (0.5, -0.289);
\coordinate (L3) at (0, 0.577);
\coordinate (L4) at (-1, -0.577);
\coordinate (L5) at (1, -0.577);
\coordinate (L6) at (0, 1.154);
\coordinate (L7) at (-1.5, -0.867);
\coordinate (L8) at (1.5, -0.867);
\coordinate (L9) at (0, 1.731);
\coordinate (L10) at (-2., -1.156);
\coordinate (L11) at (2., -1.156);
\coordinate (L12) at (0, 2.308);

\begin{scope}[every coordinate/.style={shift={(0, 0)}}]
\node at ([c]lab) {$N_{11}$};
\fill[black!100] ([c]L1) circle(0.5ex)
		 ([c]L2) circle(0.5ex)
		 ([c]L3) circle(0.5ex)
		 ([c]L4) circle(0.5ex)
		 ([c]L5) circle(0.5ex)
		 ([c]L6) circle(0.5ex)
		 ([c]L7) circle(0.5ex)
		 ([c]L8) circle(0.5ex)
		 ([c]L9) circle(0.5ex)
		 ([c]L10) circle(0.5ex)
		 ([c]L11) circle(0.5ex)
		 ([c]L12) circle(0.5ex)
		 ;
\draw[,directed] ([c]L1) -- ([c]L2) ;
\draw[,directed] ([c]L1) -- ([c]L3) ;
\draw[,directed] ([c]L1) -- ([c]L4) ;
\draw[,directed] ([c]L1) -- ([c]L5) ;
\draw[,directed] ([c]L1) -- ([c]L6) ;
\draw[,directed] ([c]L2) -- ([c]L3) ;
\draw[,directed] ([c]L2) -- ([c]L5) ;
\draw[,directed] ([c]L2) -- ([c]L6) ;
\draw[,directed] ([c]L3) -- ([c]L6) ;
\draw[,directed] ([c]L4) -- ([c]L5) ;
\draw[,directed] ([c]L4) -- ([c]L6) ;
\draw[,directed] ([c]L4) -- ([c]L7) ;
\draw[,directed] ([c]L4) -- ([c]L8) ;
\draw[,directed] ([c]L4) -- ([c]L9) ;
\draw[,directed] ([c]L5) -- ([c]L6) ;
\draw[,directed] ([c]L5) -- ([c]L8) ;
\draw[,directed] ([c]L5) -- ([c]L9) ;
\draw[,directed] ([c]L6) -- ([c]L9) ;
\draw[,directed] ([c]L7) -- ([c]L8) ;
\draw[,directed] ([c]L7) -- ([c]L9) ;
\draw[,directed] ([c]L7) -- ([c]L10) ;
\draw[,directed] ([c]L7) -- ([c]L11) ;
\draw[,directed] ([c]L7) -- ([c]L12) ;
\draw[,directed] ([c]L8) -- ([c]L9) ;
\draw[,directed] ([c]L8) -- ([c]L11) ;
\draw[,directed] ([c]L8) -- ([c]L12) ;
\draw[,directed] ([c]L9) -- ([c]L12) ;
\draw[,directed] ([c]L10) -- ([c]L11) ;
\draw[,directed] ([c]L10) -- ([c]L12) ;
\draw[,directed] ([c]L11) -- ([c]L12) ;
\end{scope}

\begin{scope}[every coordinate/.style={shift={(\x, 0*\y)}}]
\node at ([c]lab) {$N_{12}$};
\fill[black!100] ([c]L1) circle(0.5ex)
		 ([c]L2) circle(0.5ex)
		 ([c]L3) circle(0.5ex)
		 ([c]L4) circle(0.5ex)
		 ([c]L5) circle(0.5ex)
		 ([c]L6) circle(0.5ex)
		 ([c]L7) circle(0.5ex)
		 ([c]L8) circle(0.5ex)
		 ([c]L9) circle(0.5ex)
		 ([c]L10) circle(0.5ex)
		 ([c]L11) circle(0.5ex)
		 ([c]L12) circle(0.5ex)
		 ;
\draw[,directed] ([c]L1) -- ([c]L2) ;
\draw[,directed] ([c]L1) -- ([c]L3) ;
\draw[,directed] ([c]L1) -- ([c]L4) ;
\draw[,directed] ([c]L1) -- ([c]L5) ;
\draw[,directed] ([c]L1) -- ([c]L6) ;
\draw[,directed] ([c]L2) -- ([c]L3) ;
\draw[,directed] ([c]L2) -- ([c]L5) ;
\draw[,directed] ([c]L2) -- ([c]L6) ;
\draw[,directed] ([c]L3) -- ([c]L6) ;
\draw[,directed] ([c]L4) -- ([c]L5) ;
\draw[,directed] ([c]L4) -- ([c]L6) ;
\draw[,directed] ([c]L4) -- ([c]L7) ;
\draw[,directed] ([c]L4) -- ([c]L8) ;
\draw[,directed] ([c]L4) -- ([c]L9) ;
\draw[,directed] ([c]L5) -- ([c]L6) ;
\draw[,directed] ([c]L5) -- ([c]L8) ;
\draw[,directed] ([c]L5) -- ([c]L9) ;
\draw[,directed] ([c]L6) -- ([c]L9) ;
\draw[,directed] ([c]L7) -- ([c]L8) ;
\draw[,directed] ([c]L7) -- ([c]L9) ;
\draw[,directed] ([c]L7) -- ([c]L10) ;
\draw[,directed] ([c]L8) -- ([c]L9) ;
\draw[,directed] ([c]L8) -- ([c]L10) ;
\draw[,directed] ([c]L8) -- ([c]L11) ;
\draw[,directed] ([c]L9) -- ([c]L10) ;
\draw[,directed] ([c]L9) -- ([c]L11) ;
\draw[,directed] ([c]L9) -- ([c]L12) ;
\draw[,directed] ([c]L10) -- ([c]L11) ;
\draw[,directed] ([c]L10) -- ([c]L12) ;
\draw[,directed] ([c]L11) -- ([c]L12) ;
\end{scope}

\begin{scope}[every coordinate/.style={shift={(0*\x, -1*\y)}}]
\node at ([c]lab) {$N_{21}$};
\fill[black!100] ([c]L1) circle(0.5ex)
		 ([c]L2) circle(0.5ex)
		 ([c]L3) circle(0.5ex)
		 ([c]L4) circle(0.5ex)
		 ([c]L5) circle(0.5ex)
		 ([c]L6) circle(0.5ex)
		 ([c]L7) circle(0.5ex)
		 ([c]L8) circle(0.5ex)
		 ([c]L9) circle(0.5ex)
		 ([c]L10) circle(0.5ex)
		 ([c]L11) circle(0.5ex)
		 ([c]L12) circle(0.5ex)
		 ;
\draw[,directed] ([c]L1) -- ([c]L4) ;
\draw[,directed] ([c]L1) -- ([c]L5) ;
\draw[,directed] ([c]L1) -- ([c]L6) ;
\draw[,directed] ([c]L2) -- ([c]L1) ;
\draw[,directed] ([c]L2) -- ([c]L5) ;
\draw[,directed] ([c]L2) -- ([c]L6) ;
\draw[,directed] ([c]L3) -- ([c]L1) ;
\draw[,directed] ([c]L3) -- ([c]L2) ;
\draw[,directed] ([c]L3) -- ([c]L6) ;
\draw[,directed] ([c]L4) -- ([c]L7) ;
\draw[,directed] ([c]L4) -- ([c]L8) ;
\draw[,directed] ([c]L4) -- ([c]L9) ;
\draw[,directed] ([c]L5) -- ([c]L4) ;
\draw[,directed] ([c]L5) -- ([c]L8) ;
\draw[,directed] ([c]L5) -- ([c]L9) ;
\draw[,directed] ([c]L6) -- ([c]L4) ;
\draw[,directed] ([c]L6) -- ([c]L5) ;
\draw[,directed] ([c]L6) -- ([c]L9) ;
\draw[,directed] ([c]L7) -- ([c]L10) ;
\draw[,directed] ([c]L7) -- ([c]L11) ;
\draw[,directed] ([c]L7) -- ([c]L12) ;
\draw[,directed] ([c]L8) -- ([c]L7) ;
\draw[,directed] ([c]L8) -- ([c]L11) ;
\draw[,directed] ([c]L8) -- ([c]L12) ;
\draw[,directed] ([c]L9) -- ([c]L7) ;
\draw[,directed] ([c]L9) -- ([c]L8) ;
\draw[,directed] ([c]L9) -- ([c]L12) ;
\draw[,directed] ([c]L11) -- ([c]L10) ;
\draw[,directed] ([c]L12) -- ([c]L10) ;
\draw[,directed] ([c]L12) -- ([c]L11) ;
\end{scope}

\begin{scope}[every coordinate/.style={shift={(1*\x, -1*\y)}}]
\node at ([c]lab) {$N_{22}$};
\fill[black!100] ([c]L1) circle(0.5ex)
		 ([c]L2) circle(0.5ex)
		 ([c]L3) circle(0.5ex)
		 ([c]L4) circle(0.5ex)
		 ([c]L5) circle(0.5ex)
		 ([c]L6) circle(0.5ex)
		 ([c]L7) circle(0.5ex)
		 ([c]L8) circle(0.5ex)
		 ([c]L9) circle(0.5ex)
		 ([c]L10) circle(0.5ex)
		 ([c]L11) circle(0.5ex)
		 ([c]L12) circle(0.5ex)
		 ;
\draw[,directed] ([c]L1) -- ([c]L4) ;
\draw[,directed] ([c]L1) -- ([c]L5) ;
\draw[,directed] ([c]L1) -- ([c]L6) ;
\draw[,directed] ([c]L2) -- ([c]L1) ;
\draw[,directed] ([c]L2) -- ([c]L5) ;
\draw[,directed] ([c]L2) -- ([c]L6) ;
\draw[,directed] ([c]L3) -- ([c]L1) ;
\draw[,directed] ([c]L3) -- ([c]L2) ;
\draw[,directed] ([c]L3) -- ([c]L6) ;
\draw[,directed] ([c]L4) -- ([c]L7) ;
\draw[,directed] ([c]L4) -- ([c]L8) ;
\draw[,directed] ([c]L4) -- ([c]L9) ;
\draw[,directed] ([c]L5) -- ([c]L4) ;
\draw[,directed] ([c]L5) -- ([c]L8) ;
\draw[,directed] ([c]L5) -- ([c]L9) ;
\draw[,directed] ([c]L6) -- ([c]L4) ;
\draw[,directed] ([c]L6) -- ([c]L5) ;
\draw[,directed] ([c]L6) -- ([c]L9) ;
\draw[,directed] ([c]L7) -- ([c]L10) ;
\draw[,directed] ([c]L8) -- ([c]L7) ;
\draw[,directed] ([c]L8) -- ([c]L10) ;
\draw[,directed] ([c]L8) -- ([c]L11) ;
\draw[,directed] ([c]L9) -- ([c]L7) ;
\draw[,directed] ([c]L9) -- ([c]L8) ;
\draw[,directed] ([c]L9) -- ([c]L10) ;
\draw[,directed] ([c]L9) -- ([c]L11) ;
\draw[,directed] ([c]L9) -- ([c]L12) ;
\draw[,directed] ([c]L11) -- ([c]L10) ;
\draw[,directed] ([c]L12) -- ([c]L10) ;
\draw[,directed] ([c]L12) -- ([c]L11) ;
\end{scope}
\end{tikzpicture}
\caption{The extensions of $N_1$ (top row) and of $N_2$ (bottom row)}\label{N1-N2-extensions}
\end{center}
\end{figure}

$N_2$ has two possible extensions by an extra layer, but both of them can be extended to semi-transitive orientation: see $N_{21}$ and $N_{22}$ in Figure~\ref{N1-N2-extensions}. Two external layers of $N_{21}$ and $N_{22}$ form $N_2$ and $P_2$, respectively, as desired.

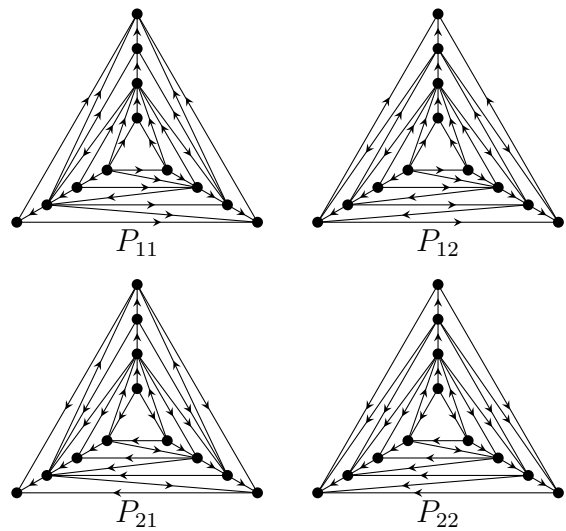
\begin{figure}[!htbp]
\begin{center}
\begin{tikzpicture}[scale=0.8,dot/.style={circle,inner sep=1pt,fill,name=#1}]
\def\dist{3.5}
\def\x{5}
\def\y{4.5}
\coordinate (lab) at (0,-1.5) ; 
\coordinate (L1) at (-0.5, -0.289);
\coordinate (L2) at (0.5, -0.289);
\coordinate (L3) at (0, 0.577);
\coordinate (L4) at (-1, -0.577);
\coordinate (L5) at (1, -0.577);
\coordinate (L6) at (0, 1.154);
\coordinate (L7) at (-1.5, -0.867);
\coordinate (L8) at (1.5, -0.867);
\coordinate (L9) at (0, 1.731);
\coordinate (L10) at (-2., -1.156);
\coordinate (L11) at (2., -1.156);
\coordinate (L12) at (0, 2.308);

\begin{scope}[every coordinate/.style={shift={(0*\x, 0*\y)}}]
\node at ([c]lab) {$P_{11}$};
\fill[black!100] ([c]L1) circle(0.5ex)
		 ([c]L2) circle(0.5ex)
		 ([c]L3) circle(0.5ex)
		 ([c]L4) circle(0.5ex)
		 ([c]L5) circle(0.5ex)
		 ([c]L6) circle(0.5ex)
		 ([c]L7) circle(0.5ex)
		 ([c]L8) circle(0.5ex)
		 ([c]L9) circle(0.5ex)
		 ([c]L10) circle(0.5ex)
		 ([c]L11) circle(0.5ex)
		 ([c]L12) circle(0.5ex)
		 ;
\draw[,directed] ([c]L1) -- ([c]L2) ;
\draw[,directed] ([c]L1) -- ([c]L3) ;
\draw[,directed] ([c]L1) -- ([c]L4) ;
\draw[,directed] ([c]L1) -- ([c]L5) ;
\draw[,directed] ([c]L1) -- ([c]L6) ;
\draw[,directed] ([c]L2) -- ([c]L3) ;
\draw[,directed] ([c]L2) -- ([c]L5) ;
\draw[,directed] ([c]L2) -- ([c]L6) ;
\draw[,directed] ([c]L3) -- ([c]L6) ;
\draw[,directed] ([c]L4) -- ([c]L5) ;
\draw[,directed] ([c]L4) -- ([c]L6) ;
\draw[,directed] ([c]L4) -- ([c]L7) ;
\draw[,directed] ([c]L5) -- ([c]L6) ;
\draw[,directed] ([c]L5) -- ([c]L7) ;
\draw[,directed] ([c]L5) -- ([c]L8) ;
\draw[,directed] ([c]L6) -- ([c]L7) ;
\draw[,directed] ([c]L6) -- ([c]L8) ;
\draw[,directed] ([c]L6) -- ([c]L9) ;
\draw[,directed] ([c]L7) -- ([c]L8) ;
\draw[,directed] ([c]L7) -- ([c]L9) ;
\draw[,directed] ([c]L7) -- ([c]L10) ;
\draw[,directed] ([c]L7) -- ([c]L11) ;
\draw[,directed] ([c]L7) -- ([c]L12) ;
\draw[,directed] ([c]L8) -- ([c]L9) ;
\draw[,directed] ([c]L8) -- ([c]L11) ;
\draw[,directed] ([c]L8) -- ([c]L12) ;
\draw[,directed] ([c]L9) -- ([c]L12) ;
\draw[,directed] ([c]L10) -- ([c]L11) ;
\draw[,directed] ([c]L10) -- ([c]L12) ;
\draw[,directed] ([c]L11) -- ([c]L12) ;
\end{scope}

\begin{scope}[every coordinate/.style={shift={(\x, 0*\y)}}]
\node at ([c]lab) {$P_{12}$};
\fill[black!100] ([c]L1) circle(0.5ex)
		 ([c]L2) circle(0.5ex)
		 ([c]L3) circle(0.5ex)
		 ([c]L4) circle(0.5ex)
		 ([c]L5) circle(0.5ex)
		 ([c]L6) circle(0.5ex)
		 ([c]L7) circle(0.5ex)
		 ([c]L8) circle(0.5ex)
		 ([c]L9) circle(0.5ex)
		 ([c]L10) circle(0.5ex)
		 ([c]L11) circle(0.5ex)
		 ([c]L12) circle(0.5ex)
		 ;
\draw[,directed] ([c]L1) -- ([c]L2) ;
\draw[,directed] ([c]L1) -- ([c]L3) ;
\draw[,directed] ([c]L1) -- ([c]L4) ;
\draw[,directed] ([c]L1) -- ([c]L5) ;
\draw[,directed] ([c]L1) -- ([c]L6) ;
\draw[,directed] ([c]L2) -- ([c]L3) ;
\draw[,directed] ([c]L2) -- ([c]L5) ;
\draw[,directed] ([c]L2) -- ([c]L6) ;
\draw[,directed] ([c]L3) -- ([c]L6) ;
\draw[,directed] ([c]L4) -- ([c]L5) ;
\draw[,directed] ([c]L4) -- ([c]L6) ;
\draw[,directed] ([c]L4) -- ([c]L7) ;
\draw[,directed] ([c]L5) -- ([c]L6) ;
\draw[,directed] ([c]L5) -- ([c]L7) ;
\draw[,directed] ([c]L5) -- ([c]L8) ;
\draw[,directed] ([c]L6) -- ([c]L7) ;
\draw[,directed] ([c]L6) -- ([c]L8) ;
\draw[,directed] ([c]L6) -- ([c]L9) ;
\draw[,directed] ([c]L7) -- ([c]L8) ;
\draw[,directed] ([c]L7) -- ([c]L9) ;
\draw[,directed] ([c]L7) -- ([c]L10) ;
\draw[,directed] ([c]L8) -- ([c]L9) ;
\draw[,directed] ([c]L8) -- ([c]L10) ;
\draw[,directed] ([c]L8) -- ([c]L11) ;
\draw[,directed] ([c]L9) -- ([c]L10) ;
\draw[,directed] ([c]L9) -- ([c]L11) ;
\draw[,directed] ([c]L9) -- ([c]L12) ;
\draw[,directed] ([c]L10) -- ([c]L11) ;
\draw[,directed] ([c]L10) -- ([c]L12) ;
\draw[,directed] ([c]L11) -- ([c]L12) ;
\end{scope}

\begin{scope}[every coordinate/.style={shift={(0*\x,-1* \y)}}]
\node at ([c]lab) {$P_{21}$};
\fill[black!100] ([c]L1) circle(0.5ex)
		 ([c]L2) circle(0.5ex)
		 ([c]L3) circle(0.5ex)
		 ([c]L4) circle(0.5ex)
		 ([c]L5) circle(0.5ex)
		 ([c]L6) circle(0.5ex)
		 ([c]L7) circle(0.5ex)
		 ([c]L8) circle(0.5ex)
		 ([c]L9) circle(0.5ex)
		 ([c]L10) circle(0.5ex)
		 ([c]L11) circle(0.5ex)
		 ([c]L12) circle(0.5ex)
		 ;
\draw[,directed] ([c]L1) -- ([c]L4) ;
\draw[,directed] ([c]L1) -- ([c]L5) ;
\draw[,directed] ([c]L1) -- ([c]L6) ;
\draw[,directed] ([c]L2) -- ([c]L1) ;
\draw[,directed] ([c]L2) -- ([c]L5) ;
\draw[,directed] ([c]L2) -- ([c]L6) ;
\draw[,directed] ([c]L3) -- ([c]L1) ;
\draw[,directed] ([c]L3) -- ([c]L2) ;
\draw[,directed] ([c]L3) -- ([c]L6) ;
\draw[,directed] ([c]L4) -- ([c]L7) ;
\draw[,directed] ([c]L5) -- ([c]L4) ;
\draw[,directed] ([c]L5) -- ([c]L7) ;
\draw[,directed] ([c]L5) -- ([c]L8) ;
\draw[,directed] ([c]L6) -- ([c]L4) ;
\draw[,directed] ([c]L6) -- ([c]L5) ;
\draw[,directed] ([c]L6) -- ([c]L7) ;
\draw[,directed] ([c]L6) -- ([c]L8) ;
\draw[,directed] ([c]L6) -- ([c]L9) ;
\draw[,directed] ([c]L7) -- ([c]L10) ;
\draw[,directed] ([c]L7) -- ([c]L11) ;
\draw[,directed] ([c]L7) -- ([c]L12) ;
\draw[,directed] ([c]L8) -- ([c]L7) ;
\draw[,directed] ([c]L8) -- ([c]L11) ;
\draw[,directed] ([c]L8) -- ([c]L12) ;
\draw[,directed] ([c]L9) -- ([c]L7) ;
\draw[,directed] ([c]L9) -- ([c]L8) ;
\draw[,directed] ([c]L9) -- ([c]L12) ;
\draw[,directed] ([c]L11) -- ([c]L10) ;
\draw[,directed] ([c]L12) -- ([c]L10) ;
\draw[,directed] ([c]L12) -- ([c]L11) ;
\end{scope}

\begin{scope}[every coordinate/.style={shift={(\x, -1*\y)}}]
\node at ([c]lab) {$P_{22}$};
\fill[black!100] ([c]L1) circle(0.5ex)
		 ([c]L2) circle(0.5ex)
		 ([c]L3) circle(0.5ex)
		 ([c]L4) circle(0.5ex)
		 ([c]L5) circle(0.5ex)
		 ([c]L6) circle(0.5ex)
		 ([c]L7) circle(0.5ex)
		 ([c]L8) circle(0.5ex)
		 ([c]L9) circle(0.5ex)
		 ([c]L10) circle(0.5ex)
		 ([c]L11) circle(0.5ex)
		 ([c]L12) circle(0.5ex)
		 ;
\draw[,directed] ([c]L1) -- ([c]L4) ;
\draw[,directed] ([c]L1) -- ([c]L5) ;
\draw[,directed] ([c]L1) -- ([c]L6) ;
\draw[,directed] ([c]L2) -- ([c]L1) ;
\draw[,directed] ([c]L2) -- ([c]L5) ;
\draw[,directed] ([c]L2) -- ([c]L6) ;
\draw[,directed] ([c]L3) -- ([c]L1) ;
\draw[,directed] ([c]L3) -- ([c]L2) ;
\draw[,directed] ([c]L3) -- ([c]L6) ;
\draw[,directed] ([c]L4) -- ([c]L7) ;
\draw[,directed] ([c]L5) -- ([c]L4) ;
\draw[,directed] ([c]L5) -- ([c]L7) ;
\draw[,directed] ([c]L5) -- ([c]L8) ;
\draw[,directed] ([c]L6) -- ([c]L4) ;
\draw[,directed] ([c]L6) -- ([c]L5) ;
\draw[,directed] ([c]L6) -- ([c]L7) ;
\draw[,directed] ([c]L6) -- ([c]L8) ;
\draw[,directed] ([c]L6) -- ([c]L9) ;
\draw[,directed] ([c]L7) -- ([c]L10) ;
\draw[,directed] ([c]L8) -- ([c]L7) ;
\draw[,directed] ([c]L8) -- ([c]L10) ;
\draw[,directed] ([c]L8) -- ([c]L11) ;
\draw[,directed] ([c]L9) -- ([c]L7) ;
\draw[,directed] ([c]L9) -- ([c]L8) ;
\draw[,directed] ([c]L9) -- ([c]L10) ;
\draw[,directed] ([c]L9) -- ([c]L11) ;
\draw[,directed] ([c]L9) -- ([c]L12) ;
\draw[,directed] ([c]L11) -- ([c]L10) ;
\draw[,directed] ([c]L12) -- ([c]L10) ;
\draw[,directed] ([c]L12) -- ([c]L11) ;
\end{scope}

\end{tikzpicture}
\caption{The extensions of $P_1$ (top row) and of $P_2$ (bottom row)}\label{P1-P2-extensions}
\end{center}
\end{figure}

$P_1$ has two possible extensions by an extra layer, but both of them can be extended to semi-transitive orientation: see $P_{11}$ and $P_{12}$ in Figure~\ref{P1-P2-extensions}. Two external layers of
$P_{11}$ and $P_{12}$ form $P_2$ and $N_2$, respectively, as desired.

Finally, $P_2$ has two possible extensions by an extra layer, but both of them can be extended to semi-transitive orientation: see $P_{21}$ and $P_{22}$ in Figure~\ref{P1-P2-extensions}. Two external layers of
$P_{21}$ and $P_{22}$ form $P_1$ and $N_1$, respectively, as desired.

We are done.

\section{Concluding remarks}\label{sec5}

It is easy to see that any triangulation of a GCCG with more than three sectors contains no $K_4$ as an induced subgraph. Using this fact, it is not difficult to see that the orientation $O$ defined in Section~\ref{sec3} contains no {\em transitively} oriented induced subgraphs on four or more vertices. That means that removing any directed edge in such an orientation, the resulting graph will avoid shortcuts (and directed cycles). As a particular case of this observation, when only diagonal edges can be removed, we have the following generalization of Theorem~\ref{thm-morethan-3}.

\begin{thm} Given a GCCG $G$ with more than three sectors, triangulate some of cells of $G$ to obtain a graph $T$. Then $T$ is word-representable if and only if it contains no $W_5$ or $W_7$ as an induced subgraph. \end{thm}

Unfortunately, such a generalization for Theorem~\ref{thm-3} does not follow directly from our proofs. Indeed, for example, the graph $N_{11}$ in Figure~\ref{N1-N2-extensions} contains a {\em transitively} oriented copy of $K_4$ (in the middle of the graph), and removing the rightmost edge in the $K_4$ we will obtain a shortcut. Thus, we leave it as an open problem to decide whether avoidance of the graphs in Figure~\ref{non-repr-induced-subgraphs} characterize triangulations of selected cells in a GCCG with three sectors, and if not then to find such a characterization.

\section*{Acknowledgments}

The work of the first and the third authors was supported by the 973 Project, the PCSIRT Project of the Ministry of Education and the National Science Foundation of China. The  second author is grateful to Bill Chen and Arthur Yang for their hospitality during the author's stay at the Center for Combinatorics at Nankai University in June-July 2015. All the authors are also grateful to the  Center for Applied Mathematics at Tianjin University for its generous support.

\end{document}